\documentclass{amsart}%
\usepackage{amsmath,amsthm}
\usepackage{amsfonts}
\usepackage{amssymb}
\usepackage{graphicx}
\usepackage[colorlinks=true,citecolor=red]{hyperref}
\usepackage[usenames,dvipsnames]{color}
\usepackage{amsmath}%
\providecommand{\U}[1]{\protect\rule{.1in}{.1in}}
\providecommand{\U}[1]{\protect\rule{.1in}{.1in}}

\newtheorem{theorem}{Theorem}[section]
\newtheorem{corollary}[theorem]{Corollary}
\newtheorem{lemma}[theorem]{Lemma}
\newtheorem{proposition}[theorem]{Proposition}
\theoremstyle{definition}
\newtheorem{remark}[theorem]{Remark}

\renewcommand{\phi}{\varphi}
\begin{document}
\title{Defining $R$ and $G(R)$}
\author{Dan Segal and Katrin Tent}

\begin{abstract}
We show that for Chevalley groups $G(R)$ of rank at least $2$ over an integral
domain $R$ each root subgroup is (essentially) the double centralizer of a
corresponding root element. In many cases, this implies that $R$ and $G(R)$
are bi-interpretable, yielding a new approach to bi-interpretability for
algebraic groups over a wide range of rings and fields.

For such groups it then follows that the group $G(R)$ is finitely
axiomatizable in the appropriate class of groups provided $R$ is finitely
axiomatizable in the corresponding class of rings.

\end{abstract}
\maketitle

\section{Introduction}

A Chevalley-Demazure group scheme $G$ assigns to each commutative ring $R$ a
group $G(R).$ If $R$ is an integral domain with field of fractions $k,$ one
can realise $G(R)$ as the group of $R$-points of $G(k)$, where $G(k)$ is taken
in a given matrix representation (see e.g. \cite{A}, \S 1). Group-theoretic
properties of $G(R)$ tend to reflect ring-theoretic properties of $R.$ In this
paper we consider properties that are expressible in first-order language;
specifically, we establish sufficient conditions for $G(R)$ to be
\emph{bi-interpretable} with $R.$ This is a slightly subtle concept, defined
in \cite{P}, def. 3.1 (cf. \cite{HMT}, Chapter 5); see \S \ref{bisection}
below. A bi-interpretation sets up a bijective correspondence between
first-order properties of the group and first-order properties of the ring.
Results of this nature for $R$ a field go back to Mal'cev \cite{M} and Zilber
\cite{Z}.

\begin{theorem}
\label{thm:bi-interpretation} Let $G$ be a simple adjoint Chevalley-Demazure
group scheme of rank at least two, and let $R$ be an integral domain. Then $R$
and $G(R)$ are bi-interpretable, assuming in case $G$ is of type $E_{6}%
,~E_{7},~E_{8},$ or $F_{4}$ that $R$ has at least two units.
\end{theorem}

For convenience, we will refer to the final assumption as `the units
condition'; it is automatically satisfied when $\mathrm{char}(R)\neq2$. The
condition is used in the proof, but may not be essential.

Throughout the paper, $G$ will denote a simple Chevalley-Demazure group scheme
defined by a root system $\Phi$ of rank at least $2$, and $R$ will be be a
commutative integral domain. $G$ is \emph{not assumed to be adjoint}; indeed
the proof yields the same result without this assumption, under the
alternative condition that $G(R)$ have \emph{finite elementary width}: that
is, there exists $N\in\mathbb{N}$ such that every element of $G(R)$ is equal
to a product of $N$ elementary root elements $x_{\alpha}(r),$ $\alpha\in\Phi,$
$r\in R$. (When referring below to Theorem \ref{thm:bi-interpretation}, we
will mean both versions of the result.)

In particular, we have (see \S \ref{sec:applications}):

\begin{corollary}
\label{corcor}The -- not necessarily adjoint -- group $G(R)$ is
bi-interpretable with $R$ in each of the following cases:

\emph{(i)} $R$ is a field;

\emph{(ii)} $G$ is simply connected, and $R$ is either a local domain or a
Dedekind domain of arithmetic type, that is, the ring of
$S$-integers in a number field $k$ w.r.t. a finite set $S$ of places of $k$.
\end{corollary}

For related results (in some ways less general, in some ways more) see
\cite{MS}, \cite{B} and \cite{AM}.

These results have consequences related to `first-order rigidity'. A group (or
ring) $X$ is \emph{first-order rigid }(or \emph{relatively axiomatizable}%
)\emph{ in a class }$\mathcal{C}$ if any member of $\mathcal{C}$ elementarily
equivalent to $X$ is isomorphic to $X$. For example, Avni, Lubotzky and Meiri
\cite{ALM} prove that all higher-rank non-uniform arithmetic groups are
first-order rigid in the class of f.g. groups.

A stronger condition is relative\emph{ finite axiomatizability}, or \emph{FA}:
$X$ is FA in $\mathcal{C}$ if there is a first-order sentence such that $X$ is
the unique member of $\mathcal{C}$ (up to isomorphism) that satisfies this
sentence. When $\mathcal{C}$ is the class of finitely generated groups, resp.
rings, the latter property is often called \emph{QFA}, or quasi-finitely
axiomatizable; see \cite{NSG}, \cite{AKNZ}, and for recent variations on this
theme \cite{NS2}. (This should not be confused with the notion of quasi finite
axiomatizability used in model theory, see e.g. \cite{P}, Chapter 3, \cite{AZ}.)

Suppose that $G(R)$ is bi-interpretable with $R$. Then $G(R)$ is first-order
rigid, resp. FA in $\mathcal{C}$ if and only if $R$ has this property relative
to $\mathcal{C}^{\prime}$, provided the `reference classes' $\mathcal{C}$ and
$\mathcal{C}^{\prime}$ are suitably chosen. In particular, in \S \ref{ax} we establish

\begin{corollary}
\label{thm:fa} Assume that $G$ and $R$ satisfy the hypotheses of Theorem
\ref{thm:bi-interpretation}. If $R$ is first-order rigid, resp. FA, in (a) the
class of finitely generated rings, (b) the class of profinite rings, (c) the
class of locally compact (or t.d.l.c.) topological rings, then $G(R)$ has the
analogous property in (a) the class of finitely generated groups, (b) the
class of profinite groups, (c) the class of locally compact (or t.d.l.c.)
topological groups.
\end{corollary}

In most cases the converse of this corollary is also valid, see \S \ref{ax}.

It is important to note that in cases (b) and (c), the first-order axioms can
\emph{a priori }only determine the group up to isomorphism \emph{as an
abstract group}, cf. \cite{NS2}, \S 1.2; in most of the cases under
consideration, this is sufficient to determine the group as a topological
group, see Proposition \ref{CSP}.

In \S \ref{sec:applications} we deduce

\begin{corollary}
\label{axiomcor}\emph{(i)} Let\emph{ }$R$ be an integral domain. If $G$ is
adjoint and the group $G(R)$ is finitely generated then $G(R)$ is FA among
f.g. groups, assuming that the units condition holds.

\emph{(ii) }Let $\mathfrak{o}_{S}$ be the ring of $S$-integers in an algebraic number field
$k,$ where $S$ is a finite set of primes of $k$. If $G$ 
is adjoint or simply connected then the $S$-arithmetic group 
$G(\mathfrak{o}_{S})$ is FA among f.g. groups.

\emph{(iii)} If $G$ is adjoint or simply connected and $R$ is one of the complete local
rings $\mathbb{F}_{q}[[t_{1},\ldots,t_{n}]]$,$~$
$\mathfrak{o}_{q}[[t_{1},\ldots,t_{n}]]$ 
($n\geq0$) then $G(R)$ is FA in the class of profinite groups.

\emph{(iv)} If $k$ is a local field then $G(k)$ is FA in the class of locally
compact groups.
\end{corollary}

(Here $\mathfrak{o}_{q}=\mathbb{Z}_{p}[\zeta],$ where $q=p^{f}$ and $\zeta$ is
a primitive $(q-1)$th root of unity). For the fact that the $S$-arithmetic
groups in (ii) are indeed finitely generated see \cite{BS}.)

Our final result moves away from integral domains. The model theory of
ad\`{e}le rings and some of their subrings has attracted some recent interest
(\cite{DM}, \cite{D}, \cite{AMO}), and it seems worthwhile to extend the
results in that direction.

Let $\mathbb{A}$ denote the ad\`{e}le ring of a global field $K$, with
$\mathrm{char}(K)\neq2,3,5$. We consider subrings of $\mathbb{A}$ of the
following kind:%
\begin{equation}
A=\mathbb{A},~~~~~A=~\prod_{\mathfrak{p}\in\mathcal{P}}\mathfrak{o}%
_{\mathfrak{p}} \label{adelic}%
\end{equation}
where $\mathfrak{o}$ is the ring of integers of $K$ and $\mathcal{P}$ may be
any non-empty set of primes (or places) of $K$. For example, $A$ could be the
whole ad\`{e}le ring of $\mathbb{Q}$, or $\widehat{\mathbb{Z}}=\prod
_{p}\mathbb{Z}_{p}.$

\begin{theorem}
\label{adgps}Let $G$ be a simple Chevalley-Demazure group scheme of rank at
least $2,$ or else one of the groups $\mathrm{SL}_{2}$, $\mathrm{SL}%
_{2}/\left\langle -1\right\rangle ,$ $\mathrm{PSL}_{2}$. Let $A$ be as in
\emph{(\ref{adelic})}. Then $A$ is bi-interpretable with the group $G(A)$.
\end{theorem}

When $\left\vert \mathcal{P}\right\vert =1,$ this is included in Theorem
\ref{thm:bi-interpretation} for groups of higher rank, and is established in
\cite{NS2}, \S 4 for groups of type $\mathrm{SL}_{2}$.

The main point of the paper is to show how results like Theorem
\ref{thm:bi-interpretation} may be deduced from the fact that \emph{root
subgroups are definable}. This in turn is a (relatively straightforward)
consequence of our main structural result.

The root subgroup of $G$ associated to a root $\alpha$ is denoted $U_{\alpha}%
$. It seems to be part of the folklore that for a field $k$, the subgroup
$U_{\alpha}(k)$ is equal to its own double centralizer in $G(k)$. We will need
a more general version of this; as we could not find a reference, and the
result for some rings is perhaps somewhat unexpected, we will present three
different approaches to the proof, each applicable to a slightly different
range of cases.

\begin{theorem}
\label{thm:double centralizer}Assume that $R$ satisfies the units condition.
Let $U$ be a root subgroup of $G$ and let $1\neq u\in U(R)$. Write $Z$ for the
centre of $G$.\emph{ }Then $\ $%
\begin{equation}
\mathrm{C}_{G(R)}\mathrm{C}_{G(R)}(u)=\mathrm{Z}(\mathrm{C}_{G(R)}%
(u))=U(R)Z(R) \label{dc1}%
\end{equation}
unless $G$ is of type $C_{n}$ (including $B_{2}=C_{2}$)$,$ $U$ belongs to a
short root $\alpha$ and $R^{\ast}=\{\pm1\}$, in which case%
\begin{equation}
\mathrm{Z}(\mathrm{C}_{G(R)}(u))\leq U(R)U_{1}(R)U_{2}(R)Z(R) \label{dc2}%
\end{equation}
where $U_{1}$ and $U_{2}$ are root subgroups belonging to long roots adjacent
to $\alpha$ in a $B_{2}$ subsystem.
\end{theorem}

In the exceptional case, $\mathrm{Z}(\mathrm{C}_{G(R)}(u))$ actually turns out
to be two-dimensional: the precise description is given in \S \ref{class}.

If one assumes that $R$ has at least four units, the theorem can be proved
very quickly, and we do this in \S \ref{dcanddefrt} below. Remaining cases are
dealt with in \S \S \ref{torus2}, \ref{building} and \ref{class}; these can be
skipped by the reader unconcerned with `difficult' rings such as $\mathbb{Z}$.

As for definability, we shall deduce

\begin{corollary}
\label{defibl}(Assuming the units condition). For each root $\alpha$ the root
subgroup $U_{\alpha}(R)$ is definable, unless possibly $G=\mathrm{Sp}_{4}(R)$,
$\mathrm{char}(R)=0$ and $R/2R$ is infinite; in any case $U_{\alpha}(R)Z(R)$
is definable.
\end{corollary}

\emph{Definable} here means `definable with parameters': a subset $H$ in a
group $\Gamma$ is \emph{definable} if there are a first-order formula
$\varphi$ and elements $g_{1},\ldots,g_{m}\in\Gamma$ such that $H=\left\{
h\in\Gamma~\mid\varphi(h,g_{1},\ldots,g_{m}\}\text{ holds}~\right\}  $.

This is good enough for the proof of Theorem \ref{thm:bi-interpretation},
which appears in \S \ref{bisection}.

\bigskip

\textbf{Remark. \ } Essentially the same proof establishes
Corollary~\ref{defibl} whenever $G$ is a $k$-isotropic algebraic group with
the maximal $k$-torus defined over $R$, provided $R$ has at least four units.
Whether the other results can be extended in this direction remains to be
seen, cf. \cite{KRT}, \cite{ALM}, \cite{AM}.

\bigskip

\textbf{Regarding Chevalley groups of rank 1}. It is easy to verify both
Theorem \ref{thm:double centralizer} and Corollary \ref{defibl} for groups $G$
of type $A_{1}$.

It is shown in \cite{NS2}, \S 4 that $\mathrm{SL}_{2}(R)$ is bi-interpretable
with $R$ if $R$ is a profinite local domain; thus Cor. \ref{axiomcor}(iii)
holds also for $G=\mathrm{SL}_{2}$.

We do not know if the other cases hold for $\mathrm{SL}_{2}$. It seems
extremely unlikely that $\mathrm{SL}_{2}(\mathbb{Z})$ can be FA or even
first-order rigid, as it is virtually free; results of Sela \cite{Se1},
\cite{Se2} concerning free products imply that $\mathrm{PSL}_{2}(\mathbb{Z})$
is not first-order rigid, and so not bi-interpretable with $\mathbb{Z}$.

In the proofs we have frequent recourse to the Chevalley commutator formula,
summarized for convenience in the Appendix.

\section{Double centralizers and definability of root groups\label{dcanddefrt}%
}

Following \cite{A} we denote by $T$ the distinguished maximal torus of $G$
determined by $\Phi$. Let $N$ denote the normalizer of $T$ in $G$, so that the
Weyl is group $W=N/T$. We will sometimes use the fact that $W$ permutes the
root subgroups, and acts transitively on the set of short roots and on the set
of long roots. Each $w\in W$ has a coset representative $n_{w}\in N(R)$ (in
fact, in the subgroup generated by root elements of the form $x_{a}(\pm1)$ ) (
\cite{C}, \S 7.2 and Lemma 6.4.4). Thus all long (resp. short) root subgroups
are conjugate in $G(R)$.

The field of fractions of $R$ will be denoted $k$, and its algebraic closure
$\overline{k}$. Sometimes we identify $G$ with $G(\overline{k})$. We write
$\pi:G\rightarrow G/Z$ for the quotient map.

\bigskip

We begin by clarifying the relation between the $R$-points of the algebraic
group $U_{\alpha}$ and the $1$-parameter group $x_{\alpha}(R)$; this is the
link between Corollary \ref{defibl} and the main theorems.

\begin{lemma}
\label{step1}Let $U=U_{\alpha}$ be a root subgroup. Then%
\begin{align}
U\cap G(R)  &  =U(R)=x_{\alpha}(R),\label{2.1i}\\
UZ\cap G(R)  &  =U(R)Z(R). \label{2.1ii}%
\end{align}

\end{lemma}

\begin{proof}
(\ref{2.1i}): If $R$ is a PID, or more generally an intersection of PIDs such
as a Dedekind ring, this follows from \cite{St}, Lemma 49(b). In the general
case, it is a consequence of the fact that the morphism $x_{\alpha}$ from the
additive group scheme to $G$ is a closed immersion (\cite{Co}, Thm. 4.1.4;
\cite{SGA} exp. XX, remark following Corollaire 5.9).

(\ref{2.1ii}): Say $g=x_{\alpha}(\xi)z\in G(R)$ where $\xi\in\overline{k}$ and
$z\in Z$. Then
\[
x_{\alpha}(\xi)\pi=g\pi\in G(R)\pi\subseteq(G/Z)(R),
\]
whence $\xi\in R$ by (\ref{2.1i}) applied to the group scheme $G/Z$. Thus
$x_{\alpha}(\xi)\in U(R)$ and so $z\in Z(R)$.
\end{proof}

\medskip

The main step in the proof of Theorem \ref{thm:double centralizer} is

\begin{lemma}
\label{mainlemma}Assume that if $G$ is of type $E_{n}$ or $F_{4}$ then
$R^{\ast}\neq1$, and if $G$ is of type $C_{n}$ then $R^{\ast}\neq\{\pm1\}$.
Then there exists a finite set $Y\subseteq\mathrm{C}_{G(R)}(U)$ such that
$\mathrm{C}_{G}(Y)\subseteq UZ$.
\end{lemma}

To deduce the main case of the theorem, observe that $Z$ is contained in
\[
V:=\mathrm{C}_{G}(\mathrm{C}_{G}(u))\leq\mathrm{C}_{G}(\mathrm{C}%
_{G(R)}(u))\leq\mathrm{C}_{G}(Y)\leq UZ.
\]

Thus if $V$ has positive dimension we have equality throughout. This is
obvious if $\mathrm{char}(k)=0$; if $G$ is of classical type, it is easy to
see in a matrix representation that $V\geq U$ (cf. \S \ref{class}). In all
other cases, the results of \cite{LT}, \cite{S1} and \cite{S2} show that
$\mathrm{dim}(V)=1.$ Now (\ref{dc1}) follows by (\ref{2.1ii}). The proof of
Theorem 1.5 for groups of type $C_{n}$ is completed in \S \ref{class}.

\bigskip

\ The slickest proof of Lemma \ref{mainlemma} uses what we call `torus
witnesses'. Let $\alpha$ and $\beta$ be linearly independent roots. A
\emph{torus witness} for $(\alpha,\beta)$ is an element $s\in T(R)$ that
centralizes $U_{\alpha}$ and acts effectively on $U_{\beta}$:%
\[
s\in\mathrm{C}_{T(R)}(U_{\alpha}),~~\mathrm{C}_{U_{\beta}}(s)=1.
\]
Note that $s$ centralizes, respectively acts effectively, on a root group
$U_{\gamma}$ if and only if it does the same to $U_{\gamma}(R)$.

In most cases we can use `{elementary torus elements'} $h_{\gamma}(t)\in
T(R),$ defined by%
\[
h_{\gamma}(t)=x_{\gamma}(t)x_{-\gamma}(-t^{-1})x_{\gamma}(t)\cdot x_{\gamma
}(1)x_{-\gamma}(-1)x_{\gamma}(1)
\]
(\cite{St}, Lemma 20, \cite{C}, Lemma 6.4.4). Now $h_{\gamma}(t)$ acts on
$U_{\beta}$ by%
\[
x_{\beta}(r)^{h_{\gamma}(t)}=x_{\beta}(t^{-A_{\gamma\beta}}r)
\]
where
\[
A_{\gamma\beta}=\frac{2(\gamma,\beta)}{(\gamma,\gamma)}\in\{0,\pm1,\pm
2,\pm3\}
\]
(see \cite{C}, p. 194).

We first deal with the case where $R$ contains at least $4$ units:

\begin{proposition}
\label{prop:torus witness with units} Assume that $\left\vert R^{\ast
}\right\vert \geq4$. Then for each pair $(\alpha,\beta)$ of linearly
independent roots there is a torus witness $s_{\alpha,\beta}$.
\end{proposition}

\begin{proof}
Let $r\in R^{\ast}$ be such that $r^{2}\neq1\neq r^{3}$. If $\beta$ is
orthogonal to $\alpha$, then we put $s_{\alpha,\beta}=h_{\beta}(r)$. Now
suppose $\alpha$ and $\beta$ are non-orthogonal. If $\alpha$ and $\beta$ span
a diagram of type $A_{2}$, then there is a root $\gamma\neq\pm\alpha,-\beta$
such that $(\alpha,\beta)\neq(\alpha,\gamma)$. In this case, the actions of
$h_{\beta}(r)$ and $h_{\gamma}(r)$ on $U_{\alpha}$ are inverse to each other
and so $s_{\alpha,\beta}=h_{\beta}(r)h_{\gamma}(r)$ is as required. If
$\alpha,\beta$ span a diagram of type $B_{2}$ or $G_{2}$, there is a root
$\gamma$ orthogonal to $\alpha$ and non-orthogonal to $\beta$ and we put
$s_{\alpha,\beta}=h_{\gamma}(r)$.
\end{proof}

\medskip

Other cases will be considered later.

\begin{proposition}
\label{prop:double centralizer} Let $\alpha$ be a positive root. Suppose that
for every positive root $\beta\neq\alpha$ there exists a torus witness
$s_{\beta}$ for $(\alpha,\beta)$. Set $Y=\{s_{\beta}~\mid~\beta\in\Phi_{+}\}$.
Then%
\[
\mathrm{C}_{G}(Y)\leq U_{\alpha}Z.
\]

\end{proposition}

\begin{proof}
We recall the Bruhat decomposition (\cite{C}, Thm. 8.4.3, \cite{St}, p. 21).
Order the positive roots as $\alpha_{1},\ldots,$ $\alpha_{m}$ and write
$U_{i}=U_{\alpha_{i}}$. For $w\in W$ put%
\[
S(w)=\left\{  ~i~\mid~w(\alpha_{i})\in\Phi_{-}\right\}
\]
where $\Phi_{-}$ is the set of negative roots. Then each element of $G$ can be
written uniquely in the form%
\begin{equation}
g=u_{1}\ldots u_{m}\cdot tn_{w}\cdot v_{1}\ldots v_{m} \label{bruhaha}%
\end{equation}
where $w\in W$, $t\in T$,$~u_{i},v_{i}\in U_{i}$ and $v_{i}=1$ unless $i\in
S(w)$.

We may suppose that $\alpha=\alpha_{1}$. For each $i\geq2$ there is a torus
witness $s_{i}\in Y$ for $(\alpha_{1},\alpha_{i})$. Now let $g\in
\mathrm{C}_{G}(Y),$ and write $g$ in the form (\ref{bruhaha}). Then for each
$j\geq2$ we have%
\[
g=g^{s_{j}}=u_{1}^{s_{j}}\ldots u_{m}^{s_{j}}\cdot tn_{w}^{s_{j}}\cdot
v_{1}^{s_{j}}\ldots v_{m}^{s_{j}}.
\]
Now $s_{j}$ fixes $u_{1}$ and $v_{1}$, and moves each non-identity element of
$U_{j}$; it also normalizes $N$ and each $U_{i}.$ It follows by the uniqueness
of expression that $u_{j}=v_{j}=1$. This holds for each $j\geq2$, and we
conclude that%
\[
g=u_{1}tn_{w}v_{1}.
\]

As $tn_{w}=v_{1}^{-1}gu_{1}^{-1}$ fixes $u\in U_{\alpha},$ but conjugates
$U_{\alpha}$ to $U_{w(\alpha)}$, it follows that $w(\alpha)=\alpha$; in
particular, $1\notin S(w)$, and so $v_{1}=1$.

It remains only to prove that $tn_{w}\in Z=\mathrm{Z}(G)$. Let $\gamma$ be a
root. If $\alpha+\gamma\notin\Phi$ then $U_{\gamma}\leq\mathrm{C}%
_{G}(U_{\alpha})$. If $\alpha+\gamma$ and $\alpha-\gamma$ are both roots then
either $2\alpha+\gamma\notin\Phi$ or $2\alpha-\gamma\notin\Phi$, and then
$U_{\alpha\pm\gamma}\leq\mathrm{C}_{G}(U_{\alpha})$. It follows that $tn_{w}$
centralizes at least one of
\[
U_{\gamma},~U_{-\gamma},~U_{\alpha\pm\gamma}.
\]
As $w(\alpha)=\alpha$ this implies that $w(\gamma)=\gamma$, and as $\gamma$
was arbitrary it follows that $w=1$. Thus $tn_{w}=t\in T$, and acts on root
subgroups in the following manner:%
\[
x_{\gamma}(\xi)^{t}=x_{\gamma}(\chi(\gamma)\xi)
\]
for a certain character $\chi$. Now $\chi$ is trivial on $\alpha$ and on one
of $\gamma,~-\gamma,~\alpha+\gamma,~\alpha-\gamma$ so it is trivial on all of
them. Thus $t$ acts trivially on every root subgroup, and so $t\in
\mathrm{Z}(G)$ as required.
\end{proof}

\medskip

The `generic case' of Theorem \ref{thm:double centralizer}, where $\left\vert
R^{\ast}\right\vert \geq4$, is now completely established.

For the remainder of this section, we will take as given the conclusion of
this theorem (in its general form), and show that it implies Cor. \ref{defibl}.

Fix a root $\alpha$, set $U=U_{\alpha}$ and fix $u\in U$, $u\neq1$. We begin with

\begin{lemma}
\medskip\label{uzdef} $U(R)Z(R)$ \emph{is a definable subgroup of} $G(R)$.
\end{lemma}

\begin{proof}
It is clear that the double centralizer of an element $u$ is definable, taking
$u$ as a parameter. So if $U$ satisfies (\ref{dc1}) we are done.

Otherwise, (\ref{dc2}) holds, $\Phi=C_{n}$ and $\alpha$ is a short root. Set
$V=\mathrm{Z}(\mathrm{C}_{G(R)}(u))$. Thus%
\[
U(R)Z(R)\leq V\leq U_{-\beta}(R)U(R)U_{2\alpha+\beta}(R)Z(R)
\]
where $\alpha,\beta$ make a pair of fundamental roots in a $B_{2}$-subsystem
of $\Phi$.

Let $g=x_{-\beta}(r)x_{\alpha}(s)x_{2\alpha+\beta}(t)z\in V$, where $z\in Z$.
The commutation relations give%
\begin{align*}
\lbrack g,x_{\alpha+\beta}(1)]  &  =x_{\alpha}(\pm r)x_{2\alpha+\beta}(\pm
r)x_{2\alpha+\beta}(\pm2s)\\
\lbrack g,x_{-\alpha-\beta}(1)]  &  =x_{-\beta}(\pm2s)x_{\alpha}(\pm
t)x_{-\beta}(\pm t)
\end{align*}

Now $g$ lies in $U(R)Z(R)$ if and only if $r=t=0$, which holds if and only if%
\begin{align*}
\lbrack g,x_{\alpha+\beta}(1)]  &  \in U_{2\alpha+\beta}(R)Z(R)\text{ \ and}\\
\lbrack g,x_{-\alpha-\beta}(1)]  &  \in U_{-\beta}(R)Z(R).
\end{align*}
As $2\alpha+\beta$ and $-\beta$ are long roots, each of the two groups on the
right is definable, as is $V$. Hence $U(R)Z(R)$ is definable in this case too.
\end{proof}

\medskip

Now we can complete the

\medskip

\textbf{Proof of Corollary \ref{defibl}. \ } If $G$ is adjoint then $Z=1$ and
$U(R)=U(R)Z(R)$ is definable, by Lemma \ref{uzdef}. This holds in particular
when $\Phi=G_{2}$ (\cite{St}, p. 23).

If $\Phi$ is not of type $A_{n},~D_{2m+1}$ or $E_{6}$ we have $Z^{2}=1$
(\emph{loc. cit.}), so in all these cases we have%
\[
U(2R)=((U(R)Z(R))^{2}%
\]
which is definable. If also $R/2R$ is finite, then $U(R)$ is the union of
finitely many cosets of $U(2R),$ and so definable with the help of a few
parameters. If $\Phi=B_{2}$ then either $G$ is adjoint or $G\cong%
\mathrm{Sp}_{4}$. If the characteristic of $R$ is odd then $2R=R$. If
$\mathrm{char}(R)=2$ and $Z^{2}=1$ then $Z=1$, and there is nothing to prove.
The case where $\mathrm{char}(R)=0$, $R/2R$ is infinite and $G\cong%
\mathrm{Sp}_{4}$ is the special case in the statement of the corollary. Thus
we may assume that $\Phi\notin\{G_{2},B_{2}\}.$

Now we separate cases. Note that if $U_{\beta}(R)$ is definable for some root
$\gamma$, then so is $U_{\gamma}(R)$ for every root $\gamma$ of the same
length as $\beta$, as these subgroups are all conjugate in $G(R)$. This will
be used repeatedly without special mention.

\medskip

\emph{Case 1}: There is a root $\beta$ such that $\alpha$ and $\beta$ make a
pair of fundamental roots in a subsystem of type $A_{2}$. Now the commutator
formula shows that%
\[
U_{\alpha+\beta}(R)=[U_{\alpha}(R)Z(R),x_{\beta}(1)],
\]
so $U_{\alpha+\beta}(R)$ is definable; and $\alpha+\beta$ has the same length
as $\alpha$.

\medskip

\emph{Case 2}: There is no such $\beta$. Then there exist roots $\beta$ and
$\gamma$ such that $\alpha,~\beta,~\gamma$ form a fundamental system of type
$B_{3}$ or $C_{3},$ with $\beta$ in the middle and of the same length as
$\gamma$. Moreover, $U_{\beta}(R)$ is definable by Case 1.

Now if $\alpha$ is short and $\beta$ is long, then $2\alpha+\beta$ is a long
root, so $U_{2\alpha+\beta}(R)$ is definable. The formula
\[
\lbrack x_{\alpha}(1),x_{\beta}(r)z]=x_{\alpha+\beta}(\pm r)x_{2\alpha+\beta
}(\pm r)
\]
($z\in Z$) shows that if $g\in U_{\alpha+\beta}(R)$ then there exist $v\in
U_{\beta}(R)Z(R)$ and $w\in U_{2\alpha+\beta}(R)$ such that $gw^{-1}%
=[x_{\alpha}(1),v]$. As%
\[
U_{\alpha+\beta}U_{2\alpha+\beta}\cap U_{\alpha+\beta}Z=U_{\alpha+\beta}%
\]
it follows that $g\in U_{\alpha+\beta}(R)$ if and only if $g\in U_{\alpha
+\beta}(R)Z(R)$ and there exist $v,~w$ as above satisfying $gw^{-1}%
=[x_{\alpha}(1),v]$. Thus $U_{\alpha+\beta}(R)$ is definable; as $\alpha
+\beta$ is short the result follows for $U_{\alpha}(R).$

Suppose finally that $\alpha$ is long and $\beta$ is short. The preceding
argument, swapping the roles of $\alpha$ and $\beta$, shows that $g\in
U_{2\beta+\alpha}(R)$ if and only if $g\in U_{2\beta+\alpha}(R)Z(R)$ and there
exist $v\in U_{\alpha}(R)Z(R)$ and $w\in U_{\beta+\alpha}(R)$ such that
$gw^{-1}=[x_{\alpha}(1),v]$. Also $U_{\beta+\alpha}(R)$ is definable becaue
$\beta+\alpha$ is short like $\beta$, and so $U_{2\beta+\alpha}(R)$ is
definable. This finishes the proof as $2\beta+\alpha$ is long like $\alpha$.

\section{Bi-interpretation\label{bisection}}

In this section we shall assume Corollary \ref{defibl} and deduce Theorem
\ref{thm:bi-interpretation}.

A bi-interpretation between $R$ and $G(R)$ has four ingredients, which we
describe in the form they occur here (which is not the most general form).
`Definability' will be in one of two first-order languages, the language
$L_{\mathrm{gp}}$ of group theory and the language $L_{\mathrm{rg}}$ of ring
theory. We set $\Gamma=G(R)$, in an attempt to avoid a forest of symbols.

\begin{enumerate}
\item An interpretation of $R$ in $\Gamma;$ in most cases, this consists in an
identification of $R$ with a definable abelian subgroup $R^{\prime}$ of
$\Gamma$ such that addition in $R^{\prime}$ is the group operation in
$\Gamma,$ and multiplication in $R^{\prime}$ is definable in $\Gamma$ (thus
the ring structure on $R^{\prime}$ is $L_{\mathrm{gp}}$ definable); in one
special case, we instead take $R^{\prime}$ to be the image in $\Gamma
/\mathrm{Z}(\Gamma)$ of a definable abelian subgroup of $\Gamma$ (the target
of an interpretation can be the quotient\emph{ }of $\Gamma$ by a definable
equivalence relation, see \cite{HMT}, \S 5.3).

\item An interpretation of $\Gamma$ in $R$; namely, for some $d\in\mathbb{N}$
an identification of $\Gamma$ with a subgroup $\Gamma^{\dag}$ of
$\mathrm{GL}_{d}(R),$ where $\Gamma^{\dag}$ is definable in $L_{\mathrm{rg}}$
(thus the group structure on $\Gamma^{\dag}$ is $L_{\mathrm{rg}}$ definable,
being just matrix multiplication);

\item An $L_{\mathrm{gp}}$ definable group isomorphism from $\Gamma$ to
$\Gamma^{\dag\prime}$, the image of $\Gamma^{\dag}$ in $\mathrm{GL}%
_{d}(R^{\prime});$

\item An $L_{\mathrm{rg}}$ definable ring isomorphism from $R$ to
$R^{\prime\dag},$ the image of $R^{\prime}$ in $\mathrm{GL}_{d}(R).$
\end{enumerate}

We assume to begin with that each root group $U_{\alpha}(R)$ is definable; the
small changes needed to deal with the exceptional case in Cor. \ref{defibl}
are indicated at the end of this section.

\subsection*{Interpreting $R$ in $G(R)$}

\begin{lemma}
\label{comp}If $U_{1},\ldots,U_{q}$ are distinct positive root subgroups then
the mapping $\pi_{1}:U_{1}(R)\ldots U_{q}(R)\rightarrow U_{1}(R)$ that sends
$u_{1}\ldots u_{q}$ to $u_{1}$ (in the obvious notation) is definable.
\end{lemma}

\begin{proof}
If $g=u_{1}\ldots u_{q}$ then%
\[
\{u_{1}\}=gU_{q}(R)\ldots U_{2}(R)\cap U_{1}(R)
\]
(cf. \cite{St} Lemma 18, Cor. 2).
\end{proof}

\begin{lemma}
\label{cabdef}Let $\alpha$ and $\beta$ be any two roots. Then the mapping%
\begin{align*}
c_{\alpha\beta}:U_{\alpha}(R)  &  \rightarrow U_{\beta}(R)\\
x_{\alpha}(r)  &  \longmapsto x_{\beta}(r)
\end{align*}
is definable.
\end{lemma}

\begin{proof}
Suppose first that $\alpha$ and $\beta$ are the same length. Then there exist
an element $w$ in the Weyl group such that $w(\alpha)=\beta$, and a
representative $n_{w}$ for $w,$ with $n_{w}\in N(R)$, such that $x_{\alpha
}(r)^{n_{w}}=x_{\beta}(\eta r)$ for all $r\in R$, where $\eta=\pm1$ (\cite{C},
lemma 7.2.1). So we can define $c_{\alpha\beta}(g)=g^{\eta n_{w}}$.

Now suppose that $\alpha$ is long and $\beta$ is short. We can find a short
root $\mu$ and a long root $\nu$ such that $\mu+\nu=\gamma$ is a short root.
The commutator formula gives (for a suitable choice of sign)%
\[
\lbrack x_{\mu}(\pm1),x_{\nu}(s)]=x_{\gamma}(s)u_{3}...u_{q}%
\]
where $u_{i}\in U_{j\mu+l\nu},$ $j+l=i$ (and $q\leq5$) , so by Lemma
\ref{comp} the map $c_{\nu\gamma}$ is definable. It follows by the first case
that $c_{\alpha\beta}=c_{\alpha\nu}c_{\nu\gamma}c_{\gamma\beta}$ is definable.

Finally if $\alpha$ is short and $\beta$ is long we have $c_{\alpha\beta
}=c_{\beta\alpha}^{-1}$.
\end{proof}

\begin{lemma}
\label{defmult} Let $\alpha$, $\beta$ and $\gamma$ be any roots. The mapping%
\begin{align*}
m_{\alpha\beta\gamma}:U_{\alpha}(R)\times U_{\beta}(R)  &  \rightarrow
U_{\gamma}(R)\\
\left(  x_{\alpha}(r),x_{\beta}(s)\right)   &  \longmapsto x_{\gamma}(rs)
\end{align*}
is definable.
\end{lemma}

\begin{proof}
By the preceding lemma we may suppose that $\alpha$ and $\gamma$ are short and
that $\gamma=\alpha+\beta.$ Then apply the same argument to the formula%
\[
\lbrack x_{\alpha}(\pm r),x_{\beta}(s)]=x_{\gamma}(rs)u_{3}...u_{q}.
\]

\end{proof}

Now we interpret $R$ in $\Gamma$ as follows: fix a root $\alpha_{0},$ set
$R^{\prime}=U_{\alpha_{0}}(R)$ and identify $r\in R$ with $r^{\prime
}=x_{\alpha_{0}}(r)$. Then $m_{\alpha_{0}\alpha_{0}\alpha_{0}}$ defines
multiplication in $R^{\prime}$. Since addition in $R^{\prime}$ is simply the
group operation, we may infer

\begin{corollary}
\label{poly}Let $f$ be a polynomial over $\mathbb{Z}$. Then the mapping
$U_{\alpha_{0}}(R)\rightarrow U_{\alpha_{0}}(R)$ given by $r^{\prime
}\longmapsto f(r^{\prime})$ is $L_{\mathrm{gp}}$ definable.
\end{corollary}

\subsection*{Interpreting $G(R)$ in $R$}

The group scheme $G$ is defined as follows (see e.g. \cite{A}, \S 1). Fix a
faithful representation of the Chevalley group $G(\mathbb{C})$ in
$\mathrm{GL}_{d}(\mathbb{C})$. The ring \linebreak$\mathbb{Z}[G]=\mathbb{Z}%
[X_{ij};i,j=1,\ldots,d]$ is the $\mathbb{Z}$-algebra generated by the
co-ordinate functions on $G,$ taken w.r.t. a suitably chosen basis for the
vector space $\mathbb{C}^{d}$. For a ring $R$ we define%
\[
G(R)=\mathrm{Hom}(\mathbb{Z}[G],R).
\]
Thus an element $g\in G(R)$ may be identified with the matrix $(X_{ij}(g))$,
and the group operation is given by matrix multiplication.

Let $T_{ij}$ be independent indeterminates. The kernel of the obvious
epimorphism $\mathbb{Z}[\mathbf{T}]\rightarrow\mathbb{Z}[G]$ is an ideal,
generated by finitely many polynomials $P_{l}(\mathbf{T}),$ $l=1,\ldots,s$
say. For a matrix $g=(g_{ij})\in\mathrm{M}_{d}(R),$ we have%
\begin{equation}
g\in G(R)\Longleftrightarrow P_{l}(g_{ij})=0~~(l=1,\ldots,s). \label{Zariski}%
\end{equation}
Thus $G(R)$ is $L_{\mathrm{rg}}$ definable as a subset of $\mathrm{M}_{d}(R)$.

\subsection*{Definable isomorphisms}

To complete \textbf{Step 3}, we exhibit a definable isomorphism $\theta
:G(R)\rightarrow G(R^{\prime})\subseteq\mathrm{M}_{d}(R^{\prime})$. The
definition of such a $\theta$ is obvious; the work is to express this
definition in first-order language.

We recall the construction of $G(R)$ in more detail (cf \cite{St}, Chapters 2
and 3). For each root $\alpha$ there is a matrix $X_{\alpha}\in\mathrm{M}%
_{d}(\mathbb{Z})$ such that%
\begin{equation}
x_{\alpha}(r)=\exp(rX_{\alpha})=1+rM_{1}(\alpha)+\ldots+r^{q}M_{q}%
(\alpha)~~\ (r\in R) \label{root_matrix}%
\end{equation}
where $M_{i}(\alpha)=X_{\alpha}^{i}/i!$ has integer entries, and $q$ is fixed
(usually $q\leq2$).

We have chosen a root subgroup\ $U_{0}=U_{\alpha_{0}}(R)$ and identified it
with the ring $R$ by $r\longmapsto r^{\prime}=x_{\alpha_{0}}(r)$. We have
identified $\Gamma=G(R)$ with a group of matrices. Now define $\theta
:\Gamma\rightarrow\mathrm{M}_{d}(R^{\prime})=U_{0}^{d^{2}}\subseteq
\Gamma^{d^{2}}$ by%
\[
g\theta=(g_{ij}^{\prime}).
\]
Giving $R^{\prime}$ the ring structure inherited from $R,$ this map is
evidently a group isomorphism from $\Gamma$ to its image in $\mathrm{GL}%
_{d}(R^{\prime})$.

\begin{lemma}
\label{def_theta}For each root $\alpha$ the restriction of $\theta$ to
$U_{\alpha}(R)$ is definable.
\end{lemma}

\begin{proof}
Let $\alpha$ be a root, fix $i$ and $j,$ and write $\theta_{ij}$ for the map
$g\longmapsto g_{ij}^{\prime}$. Let $m_{l}$ denote the $(i,j)$ entry of the
matrix $M_{l}(\alpha).$ Then for $g=x_{\alpha}(r)$ we have%
\[
g\theta_{ij}=(1+m_{1}r+\ldots+m_{q}r^{q})^{\prime}.
\]
As $r^{\prime}=x_{\alpha_{0}}(r)=gc_{a\alpha_{0}},$ it follows from Cor.
\ref{poly} that the restriction of $\theta_{ij}$ to $U_{\alpha}(R)$ is
definable, and as this holds for all $i,~j$ it establishes the claim.
\end{proof}

\medskip

Say the roots are $\alpha_{1},\ldots,\alpha_{q}.$ For a natural number $N$ put%
\[
X_{N}=\left(  \prod_{i=1}^{q}U_{\alpha_{i}}(R)\right)  \cdot\ldots\cdot\left(
\prod_{i=1}^{q}U_{\alpha_{i}}(R)\right)
\]
with $N$ factors. 
Thus $X_{N}$ is a definable set, every product of $N$ elementary root elements
lies in $X_{N},$ and the preceding lemma implies that the restriction of
$\theta$ to $X_{N}$ is definable.

If $G(R)$ has finite elementary width $N$ then $G(R)=X_{N}$ and so $\theta$ is definable.

Suppose alternatively that $G$ is adjoint. Then
\begin{equation}
\bigcap_{i=1}^{q}\mathrm{C}_{G}(x_{\alpha_{i}}(1))=\mathrm{Z}(G)=1,
\label{centre}%
\end{equation}
(see Lemma \ref{mainlemma} and the discussion following it).

We quote

\begin{lemma}
\emph{(\cite{Sp}, Cor. 5.2)} There exists $L\in\mathbb{N}$ such that for each
root $\alpha$ and every $g\in G(R)$ the commutator $[x_{\alpha}(1),g]$ is a
product of $3L$ elementary root elements.
\end{lemma}

Taking $N=3L+1$ we see that each $x_{\alpha}(1)^{g}\in X_{N}$. Set
$v_{i}=x_{\alpha_{i}}(1)\theta.$ Now let $g\in G(R)$ and $h\in G(R^{\prime}).$
If $g\theta=h$ then for $i=1,\ldots,q$ there exists $x_{i}\in X_{N}$ such that%
\begin{align*}
x_{\alpha_{i}}(1)^{g}  &  =x_{i}\\
x_{i}\theta &  =v_{i}^{h}.
\end{align*}
Conversely, if this holds then $v_{i}^{h}=v_{i}^{g\theta}$ for each $i,$ so
$g\theta\cdot h^{-1}$ centralizes each $v_{i};$ as $\theta$ is an isomorphism
it follows from (\ref{centre}) that $g\theta=h$. Thus the statement
`$g\theta=h$' is expressible by a first-order formula, and $\theta$ is definable.

\medskip

To complete \textbf{Step 4}, define $\psi:R\rightarrow U_{0}\subseteq
\mathrm{M}_{d}(R)$ by $r\psi=r^{\prime}=x_{\alpha_{0}}(r)$. This is a ring
isomorphism by definition, when $U_{0}$ is given the appropriate ring
structure. The expression (\ref{root_matrix}) now implies

\begin{lemma}
\label{def_psi}The map $\psi$ is $L_{\mathrm{rg}}$ definable.
\end{lemma}

\subsection*{When $U_{\alpha}(R)$ is not definable}

Set $K=\mathrm{Z}(\Gamma)$ and write $\symbol{126}:\Gamma\rightarrow\Gamma/K$
for the quotient map. Corollary \ref{defibl} shows that each of the subgroups
$U_{\alpha}(R)K$ is definable. Lemmas \ref{comp} - \ref{defmult} remain valid,
with essentially the same proofs, if each $U_{\alpha}(R)$ is replaced by
$U_{\alpha}(R)K$. As $U_{\alpha}(R)\cap K=1$ the map $\symbol{126}$ restricts
to an isomorphism $U_{\alpha}(R)\rightarrow\widetilde{U_{\alpha}%
(R)K}=\widetilde{U_{\alpha}(R)}$, and we define $R^{\prime}%
:=\widetilde{U_{\alpha_{0}}(R)},$ setting $r^{\prime}=\widetilde{x_{\alpha
_{0}}(r)}$. Then Corollary \ref{poly} remains valid if $U_{\alpha_{0}}(R)$ is
replaced by $\widetilde{U_{\alpha_{0}}(R)}$.

The interpretation of $\Gamma$ in $R$ is as above.

We have a definable ring isomorphism $\psi:R\rightarrow\widetilde{U_{0}}$ as
in Lemma \ref{def_psi}.

Similarly, the group isomorphism $\theta:\Gamma\rightarrow\mathrm{M}%
_{d}(R^{\prime})=\widetilde{U_{0}}^{d^{2}}\subseteq\widetilde{\Gamma}^{d^{2}}$
is definable: in the proof of Lemma \ref{def_theta}, we replace each $U_{i}$
by $U_{i}K,$ and then apply the map $\symbol{126}$ to each root element that
appears in the discussion.

The bi-interpretability of $\Gamma$ with $R$ is now established in all cases.

\section{Axiomatizability\label{ax}}

In \S \ref{bisection}\ we set up a bi-interpretation of a specific shape
between a group $\Gamma$ and a ring $R,$ spelt out explicitly in points 1. -
4. at the beginning of the section. As is well known, this implies a close
correspondence between first-order properties of the two structures; here we
explore some of the consequences (professional model theorists are invited to
skip the next few paragraphs!)

The interpretation of $R$ in $\Gamma$ involves two or three formulae: one, and
if necessary two, define the subset (it was $U_{\alpha}(R)$), or its quotient
($U_{\alpha}(R)\mathrm{Z}(\Gamma)/\mathrm{Z}(\Gamma)$), that we called
$R^{\prime};$ the third defines a binary operation $m$ on $R^{\prime}$. \ Let
$P_{1}$ be a sentence that expresses the facts

\begin{enumerate}
\item each of the definable mappings denoted $\pi_{1}$ in Lemma \ref{comp}
actually is a well defined mapping

\item the definition of $m$ does define a binary operation on the set
$R^{\prime}$

\item $(R^{\prime},+,m)$ is a commutative integral domain, where $+$ is the
group operation inherited from $\Gamma$.
\end{enumerate}

Let us call this ring $A_{\Gamma}$.

The sentence $P_{1}=P_{1}(\mathbf{g})$ involves some parameters $g_{1},\ldots
g_{r}$ from $G(R).$ Let $P_{1}^{\prime}$ denote the sentence $\exists
h_{1},\ldots,h_{r}.P_{1}(\mathbf{h})$. We shall use this convention for other
sentences later.

Now if $H$ is any group that satisfies $P_{1}^{\prime},$ the same formulae
define a ring $A_{H}$. For each $L_{\mathrm{rg}}$ formula $\alpha$ there is an
$L_{\mathrm{gp}}$ formula $\alpha^{\ast}$ such that $A_{H}\models\alpha$ iff
$H\models\alpha^{\ast},$ since ring operations in $A_{H}$ are expressible in
terms of the group operation in $H$. (Note that $\alpha^{\ast}$ will involve
parameters, obtained by substituting $h_{i}$ for $g_{i}.$)

Analogously, the equations on the right-hand side of (\ref{Zariski}) may be
expressed as a formula in $L_{\mathrm{rg}},$ that for any ring $S$ defines a
subset $G(S)$ of $S^{d^{2}}$; and if $S$ is an integral domain, the set $G(S)$
with matrix multiplication is a group. For each $L_{\mathrm{gp}}$ formula
$\beta$ there is an $L_{\mathrm{rg}}$ formula $\beta^{\dag}$ such that
$G(S)\models\beta$ iff $S\models\beta^{\dag}$.

Now in \S \ref{bisection} we give (i) an $L_{\mathrm{gp}}$ formula that
defines a group isomorphism $\theta:\Gamma\rightarrow G(A_{\Gamma})$, and (ii)
an $L_{\mathrm{rg}}$ formula that defines a ring isomorphism $\psi
:R\rightarrow A_{G(R)}$. The assertions that these formulae actually define
such isomorphisms can be expressed by (i) an $L_{\mathrm{gp}}$ sentence
$P_{2}$ and (ii) an $L_{\mathrm{rg}}$ sentence $P_{3}$, say.

The results of \S \ref{bisection} amount to this: if the group $G$ and the
ring $R$ satisfy the hypotheses of Theorem \ref{thm:bi-interpretation}, then
$G(R)$ satisfies the conjunction of $P_{1}^{\prime}$ and $P_{2}^{\prime}$, and
$R$ satisfies $P_{3}^{\prime}$, where $P_{3}^{\prime}$ is obtained from
$P_{3}$ by adding an existential quantifier over the (ring) variables
corresponding to the matrix entries of the original parameters $g_{i}$.

The correspondence $\alpha\rightarrow\alpha^{\ast}$ implies that any ring
axioms satisfied by $R$ can be expressed as properties of the group
$\Gamma=G(R)$.\ If these axioms happen to determine the ring up to
isomorphism, the existence of $\theta$ then shows that the corresponding
properties of $\Gamma$, in conjunction with $P_{1}^{\prime}$ and
$P_{2}^{\prime}$, determine $\Gamma$ up to isomorphism. In the same way, if
$G(R)$ happens to be determined by some family of group axioms, then a
corresponding family of ring properties, together with $P_{3}$, will determine
$R$.

To apply this observation we need

\begin{proposition}
\label{RGprops}\emph{(i)} If $G(R)$ is a finitely generated group then $R$ is
a finitely generated ring.

\emph{(ii) }If $G(R)$ is a Hausdorff topological group then $R$ is a Hausdorff
topological ring, and $R$ is profinite, locally compact or t.d.l.c. if $G(R)$
has the same property.
\end{proposition}

\begin{proof}
(i) Suppose $G=\left\langle g\,_{1},\ldots,g_{m}\right\rangle $. The entries
of the matrices $g_{i}^{\pm1}$ generate a subring $S$ of $R$, and then
$G(R)=G(S).$ Choose a root $\alpha$. Then
\[
U_{\alpha}(R)=U_{\alpha}(k)\cap G(R)=U_{\alpha}(k)\cap G(S)=U_{\alpha}(S).
\]
As the map $r\longmapsto x_{a}(r)$ is bijective it follows that $R=S$.

(ii) Suppose that $G(R)$ is a (Hausdorff) topological group. Let
$U_{0}=U_{\alpha_{0}}$ be the root group discussed in \S \ref{bisection}. Then
$U_{0}(R)$ is closed in the topology, by Lemma \ref{step1}. Thus with the
subspace topology $U_{0}(R)$ is a topological group; it is locally compact,
compact or totally disconnected if $G(R)$ has the same property.

We have seen that $R$ is isomorphic to a ring $R^{\prime},$ where the additive
group of $R^{\prime}$ is $U_{0}(R)$. It remains to verify that the ring
multiplication in $R^{\prime}$ is continuous. This in turn follows from the
facts (a) the commutator defines a continuous map $G(R)\times G(R)\rightarrow
G(R)$ and (b) the projection mapping $\pi_{1}$ described in Lemma \ref{comp}
is continuous, because $U_{1}(R)\ldots U_{q}(R)$ is a topological direct product.
\end{proof}

\medskip We have stated the proposition for $G(R)$ for the sake of clarity.
However a more general version is required:

\begin{proposition}
\label{gvr2}Let $H$\ be a group that satisfies $P_{1}^{\prime}$ and
$P_{2}^{\prime}$, and put $S=A_{H}$. Then \emph{(i)} and \emph{(ii)} of
Proposition \ref{RGprops} hold with $S$ in place of $R$ and $H$ in place of
$G(R)$.
\end{proposition}

\begin{proof}
(i) $P_{1}$ and $P_{2}$ ensure that $S$ is a commutative integral domain and
that $H\cong G(S)$. Now the result follows from the preceding proposition.

(ii) We have $S=U$ (or $S=UZ/Z)$ where $U$ (or $UZ)$ is defined as a double
centralizer (or similar, cf. Lemma \ref{uzdef}) in $H$ (and $Z=\mathrm{Z}(H)$
). It follows that $U$ (or $UZ$) is closed in the topology of $H$. Thus $S$
inherits a topology, which makes $(S,+)$ a topological group with the given
properties. The continuity of multiplication follows as before: the assumption
that the mapping $\pi_{1}$ is well defined implies that the corresponding
product of definable subgroups is actually a topological direct product, and
hence that $\pi_{1}$ is continuous; the other ingredients in the definition of
multiplication are clearly continuous.
\end{proof}

Now we can deduce Corollary \ref{thm:fa}, in a slightly more general form.

\begin{theorem}
\label{prth}Assume that $G$ and $R$ satisfy the hypotheses of Theorem
\ref{thm:bi-interpretation}. Let $\Sigma$ be a set of sentences of
$L_{\mathrm{rg}}$ such that $R\models\Sigma$. Then there is a set
$\widetilde{\Sigma}$ of sentences of $L_{\mathrm{gp}},$ finite if $\Sigma$ is
finite, such that $G(R)\models\widetilde{\Sigma}$ and such that

\emph{(i)} Suppose that $G(R)$ is a finitely generated group. If $R$ is the
unique f. g. ring (up to isomorphism) satisfying $\Sigma$ then $G(R)$ is the
unique f. g. group (up to isomorphism) that satisfies $\widetilde{\Sigma}.$

\emph{(ii) }If $R$ is the unique profinite, locally compact, or t.d.l.c. ring
(up to isomorphism) satisfying $\Sigma$ then $G(R)$ is the unique profinite,
locally compact, or t.d.l.c. group (up to isomorphism) that satisfies
$\widetilde{\Sigma}.$
\end{theorem}

\begin{proof}
For each $\sigma\in\Sigma$ there is a formula $\sigma^{\ast}$ such that for
any group $H$ that satisfies $P_{1}^{\prime},$ we have $H\models\sigma^{\ast}$
iff $A_{H}\models\sigma$. We take $\widetilde{\Sigma}=\Sigma^{\ast}\cup
\{P_{1}^{\prime},P_{2}^{\prime}\}.$ The result now follows from Proposition
\ref{gvr2} by the preceding discussion.
\end{proof}

\textbf{Remark.} Theorem \ref{prth} has a converse, in most cases. \emph{If
}$G(R)$ \emph{is axiomatizable (or F.A.) among groups that are profinite, l.c.
or t.d.l.c. then} $R$ \emph{is similarly axiomatizable in the corresponding
class of rings}. The proof is the same, using a suitable analogue of
Proposition \ref{gvr2} (ii): in this case, it is easy to see that for a ring
$S$, the group $G(S)\subseteq\mathrm{M}_{d}(S)$ defined by the polynomial
equations (\ref{Zariski}) inherits an appropriate topology from $S$.

We are not entirely sure whether the analogue of (i) holds in all cases.
Assume that $G(R)$ is generated by its root subgroups, and \emph{either} (i)
the root system $\Phi$ is simply laced \emph{or} (ii) $\left\vert
R/2R\right\vert $ is finite and $\Phi\neq G_{2}$ \emph{or} (iii) $\left\vert
R/6R\right\vert $ is finite. Then using the idea of Lemma \ref{defmult} one
can show that if $R$ is finitely generated as a ring then $G(R)$ is a finitely
generated group. Thus we can assert: \emph{let} $R$ \emph{be a f.g. integral
domain and assume (i), (ii) or (iii). If } $G(R)$ \emph{is first-order rigid,
resp. F.A., among f.g. groups, then} $R$ \emph{has the same property among
f.g. rings}.\medskip\medskip

\subsection*{Topological vs. algebraic isomorphism}

In Theorem \ref{prth}, the phrase `up to isomorphism' refers to isomorphism as
\emph{abstract groups}. {}In part (ii), to infer that $G(R)$ is first-order
rigid, or FA, in the appropriate class of topological groups, one needs to
show that abstract isomorphism with $G(R)$ implies topological isomorphism. In
most of the cases under discussion, this is true.

A `local field' means one with a non-discrete locally compact topology, and a
\emph{locally compact} group means one that is not discrete.

\begin{proposition}
\label{CSP}\emph{(i)} Let $k$ be a local field. Then any locally compact group
abstractly isomorphic to $G(k)$ is topologically isomorphic to $G(k).$

\emph{(ii)} Let $R$ be a complete local domain with finite residue field
$\kappa$, and assume that $G$ is simply connected. Then any profinite group
abstractly isomorphic to $G(R)$ is topologically isomorphic to $G(R)$, unless
possibly $\mathrm{char}(\kappa)=2$ and $G$ is of type $B_{n}$ or $C_{n}$, or
$\mathrm{char}(\kappa)=3$ and $G$ is of type $G_{2}$.
\end{proposition}

\begin{proof}
(i) This is equivalent to the claim that $G(k)$ is determined up to
topological isomorphism by its algebraic structure.

The Bruhat decomposition of $G(k)$ is algebraically determined (e.g. by the
proof of Corollary \ref{defibl}), and it expresses $G(k)$ as a finite union of
products of copies of $k$ (the root subgroups) and of $k^{\ast}$ (the torus).
It follows that any topology on $G$ is determined by its restriction to the
root sugroups, identified with $k$. It follows from Lemma \ref{defmult} that
the algebraic stucture of $k$ is determined by that of $G$. Now a local field
that is algebraically isomorphic to $k$ is topologically isomorphic to $k$:
this is clear from the classification of local fields, see e.g. \cite{W},
Chapter 1.

In many cases a stronger result holds: \emph{every isomorphism with }$G(k)$
\emph{is continuous. }This holds when $k\neq\mathbb{C}$ (it may be deduced
from \cite{St}, Lemma 77; cf. \cite{BT}, \S 9), but obviously not for
$k=\mathbb{C}$.

(ii) This follows from the \emph{congruence subgroup property}: if $K$ is a
normal subgroup of finite index in $G(R)$ then $K$ contains the congruence
subgroup $\ker(G(R)\rightarrow G(R/I))$ for some ideal $I$ of finite index in
$R,$ see \cite{A}, Theorem 1.9. Thus every subgroup of finite index in $G(R)$
is open. Hence if $f:G(R)\rightarrow H$ is an isomorphism, where $H$ is a
profinite group, then $f^{-1}(K)$ is open in $G(R)$ for every open subgroup
$K$ of $H$, so $f$ is continuous; and a continuous isomorphism between
profinite groups is a homeomorphism.

Alternatively, it follows from \cite{LS}, Cor. 3.4 that $G(R)$ is finitely
generated as a profinite group provided the Lie algebra over $\kappa$
associated to $\Phi$ is perfect. As $G(R)$ in this case is virtually a pro-$p$
group, this in turn implies that every subgroup of finite index is open
(\cite{DDMS}, Theorem 1.17).
\end{proof}

\section{Applications}

\label{sec:applications}As before, $G$ is a simple Chevalley-Demazure group
scheme defined by a root system $\Phi$ of rank at least $2$ and $R$ is a
commutative integral domain.

The group $G(R)$ has finite elementary width in the following cases: 

\begin{enumerate}
\item When $R$ is a field, by the Bruhat decomposition (\cite{C}, Thm. 8.4.3,
\ \cite{St}, Cor. 1 on p. 21).

\item When $R$ is a local ring and $G$ is simply connected, by a theorem of
Abe, \cite{A} Proposition 1.6, together with \cite{HSVZ}, Corollary 1.

\item When\emph{ }$R$ is a Dedekind domain of arithmetic type and $G$ is
simply connected, by a theorem of Tavgen, \cite{T} \ Theorem A.
\end{enumerate}

To apply Theorem \ref{prth}, we need to pick out from this list those
rings that are also FA. Now \cite{AKNZ}, Proposition 7.1 says that \emph{every
f.g. commutative ring }is FA in the class of f.g. rings; it is shown in
\cite{NS2}, Theorem 4.4 that \emph{every regular, unramified complete local
ring with finite residue field is FA }in the class of profinite rings. (These
rings are $\mathbb{F}_{q}[[t_{1},\ldots,t_{n}]]$,$~\mathfrak{o}_{q}%
[[t_{1},\ldots,t_{n}]]$ , $n\geq0,$ where $\mathfrak{o}_{q}=\mathbb{Z}%
_{p}[\zeta],$ $q=p^{f}$, $\zeta$ a primitive $(q-1)$th root of unity).

It is also the case that every locally compact field is FA in the class of all
locally compact rings. We are grateful to Matthias Aschenbrenner for supplying
the proof of Proposition \ref{local} sketched below.

Thus we may deduce -- invoking Proposition \ref{RGprops} (i) for part (i) --

\begin{corollary}
\label{corfinal}\emph{(i) }If\emph{ }$G(R)$ is finitely generated and $G$ is
adjoint then $G(R)$ is FA among f.g. groups (assuming that the units condition holds).

\emph{(ii) If }$S$ is a finite set of primes in an algebraic number field and
$R$ is the ring of $S$-integers then the $S$-arithmetic group $G(R)$ is FA
among f.g. groups, assuming that $G$ is simply connected.

\emph{(iii) }The profinite groups $G(R)$, $R=\mathbb{F}_{q}[[t_{1}%
,\ldots,t_{n}]]$ or $R=\mathfrak{o}_{q}[[t_{1},\ldots,t_{n}]],$ $n\geq0$, are
FA among profinite groups, if $G$ is adjoint or simply connected.

\emph{(iv) If }$k$ is a local field then $G(k)$ is FA among locally compact groups.
\end{corollary}

\begin{proposition}
\label{local}\emph{(M. Aschenbrenner) }Let $k$ be a locally compact field.
Then $k$ is determined up to isomorphism within the class of locally compact
rings by finitely many first-order sentences.
\end{proposition}

\begin{proof}
The first axiom asserts that $k$ is a field. Now we consider the cases.

\textbf{1}. If $k=\mathbb{R}$, then $k$ is axiomatized by saying that $k$ is
Euclidean, that is, \emph{(a)} $-1$ is not of the form $x^{2}+y^{2}$ and
\emph{(b)} for every $x\in k$ either $x$ or $-x$ is a square. (This implies
that $k$ is an ordered field for a (unique) ordering whose set of nonnegative
elements is given by the squares; and no other local field is orderable.)

\textbf{2}. If $k=\mathbb{C}$, then $k$ is axiomatized by saying that every
element is a square.

\textbf{3.} \ Let $k=\mathbb{F}_{q}((t))$ where $q$ is a power of a prime $p.$
Ax provides in \cite{Ax} a formula $\phi_{p}$ that defines the valuation ring
in any henselian discretely valued field of residue characteristic $p$. We can
then make a sentence which expresses that the characteristic of the field is
$p$ and the residue field of the valuation ring defined by $\phi_{p}$ has size
$q$. This sentences determines $k$ up to isomorphism among all local fields.

\textbf{4.} The remaining case is where $k$ is a finite extension of
$\mathbb{Q}_{p}$. Then we use Ax's formula $\phi_{p}$ again to express that
the ramification index and residue degree of $k$ have given values $e$ and
$f$. Then $(k:\mathbb{Q}_{p})=ef$. Let $h$ be the minimal polynomial of a
primitive element for $k$ over $\mathbb{Q}_{p},$ and let $g\in\mathbb{Q}[t]$,
of degree $ef=\deg(h),$ have coefficients sufficiently close to those of $h$
that Krasner's Lemma applies, i.e. $g$ has a zero $\beta\in k$ and
$k=\mathbb{Q}_{p}(\beta)$. Then $k$ is determined among local fields by:
$p\neq0$; the formula $\phi_{p}$ defines in $k$ a valuation ring with residue
field of characteristic $p$, ramification index $e$, and residue degree $f$;
and the polynomial $g$ has a zero in $k$.
\end{proof}

\section{Torus witnesses in some exceptional groups\label{torus2}}

Returning to the proof of Lemma \ref{mainlemma}, begun in \S \ref{dcanddefrt},
we now establish the existence of the required torus witnesses for some
exceptional groups, under the blanket assumption that $R^{\ast}\neq\{1\}$. A
similar approach works for the other groups as well, but different methods
will enable us in \S \S \ref{building} and \ref{class}\ to dispense with any
conditions on $R^{\ast}$.

We begin with the following basic observation:

\begin{lemma}
\label{lem:gamma witness} Suppose that $\Phi$ is a root system of rank at
least $2$ and $r\in R^{\ast}\setminus\{1\}$. If $\alpha,\beta\in\Phi$ and
$\gamma$ is orthogonal to $\alpha$ and non-orthogonal to $\beta$, then
$s_{\alpha,\beta}=h_{\gamma}(r)$ is a torus witness for $(\alpha,\beta)$
\textbf{unless} $A_{\gamma\beta}=\pm2$ and $r=-1$ or $A_{\gamma\beta}=\pm3$
and $r^{3}=1$.
\end{lemma}

Note also that if $\mathrm{char}(R)\neq2$, $\alpha,\beta$ are non-orthogonal
and $A_{\alpha\beta}\neq2$, then $s_{\alpha,\beta}=h_{\alpha}(-1)$ is a torus witness.


\begin{lemma}
\label{lem:gamma} Let $\Phi\in\{ E_{6}, E_{7}, E_{8}, F_{4}\}$ and suppose
$\alpha, \beta\in\Phi$ are orthogonal. Then there is a root $\gamma$
orthogonal to $\alpha$ and non-orthogonal to $\beta$ unless $\Phi=F_{4}$, and
$\alpha, \beta$ are both long.
\end{lemma}

\begin{proof}
\textbf{$\Phi= E_{n}, n= 6, 7, 8$.} Let $a_{1},\ldots, a_{n}, n\in\{6, 7, 8\}$
be a set of fundamental roots where $a_{n-3}$ is the branching point. We may
assume that $\alpha=a_{1}$. If $\beta$ does not involve $a_{2}$, we can choose
$\gamma$ as a root in the subdiagram spanned by $a_{3},\ldots, a_{n}$ and
non-orthogonal to $\beta$.

Now suppose $\beta$ is a positive root involving $a_{2}$ and orthogonal to
$\alpha$. If there is a fundamental root $a_{i}, 3\leq i\leq n,$ which is
non-orthogonal to $\beta$, put $\gamma=a_{i}$. Otherwise an easy calculation
(starting from $a_{n}$) shows that for $n=7,8$ we have:
\[
\beta=\epsilon_{n}(a_{1}+2a_{2} + 2a_{n-5} + \frac{5}{2}a_{n-4}+ \frac{3}%
{2}a_{n-2} +3a_{n-3}+2a_{n-1}+a_{n})
\]

whereas for $n=6$ we must have%

\[
\beta=\epsilon_{6}(\frac{5}{4} a_{1}+\frac{5}{2} a_{2} + \frac{3}{2}a_{4}
+3a_{3}+2a_{5}+a_{6}).
\]

In either case, $\beta$ is non-orthogonal to $a_{2}$. For $n=6, 7,8$ let
\[
\gamma=a_{1}+2(a_{2}+\ldots+ a_{n-3})+a_{n-2}+a_{n-1}.
\]
Then $\gamma$ is orthogonal to $\alpha$, but not to $\beta$.

\medskip\noindent\textbf{$\Phi= F_{4}$.} Let $a_{1},\ldots, a_{4}$ be a set of
fundamental roots where $a_{1}$ is long and $a_{4}$ is short. First assume
that $\alpha=a_{1}$. If $\beta$ is short and orthogonal to $\alpha$ then
either it is contained in the subdiagram spanned by $a_{3}, a_{4}$ and we
choose $\gamma$ in this $A_{2}$-subdiagram non-orthogonal to $\beta$. Or else
we have $\beta\in\{a_{1}+2a_{2}+2a_{3}+a_{4}, a_{1}+2a_{2}+3a_{3}+a_{4},
a_{1}+2a_{2}+3a_{3}+2a_{4}\}$ and $\gamma=a_{4}$ or $\gamma=a_{3}+a_{4}$ is as required.

Next assume $\alpha=a_{4}$ is short. If $\beta$ is a positive root orthogonal
to $\alpha$ and contained in the subdiagram spanned by $a_{1}, a_{2}$, then we
find $\gamma$ as before. Otherwise we have $\beta\in\{a_{2}+2a_{3}+a_{4},
a_{1}+a_{2}+2a_{3}+a_{4}, a_{1}+2a_{2}+2a_{3}+a_{4}\}$ and $\gamma=a_{1}$ or
$\gamma=a_{1}+a_{2}$ is as required.
\end{proof}

\begin{lemma}
\label{lem:gamma nonorthogonal} Let $\Phi\in\{E_{6},E_{7},E_{8},F_{4}\}$ and
suppose $\alpha,\beta\in\Phi$ are non-orthogonal and $R^{\ast}\neq\{1\}$. Then
there is a torus witness for $(\alpha,\beta)$.
\end{lemma}

\begin{proof}
If $\alpha,\beta$ are non-orthogonal, then (replacing $\beta$ by $-\beta$ if
necessary) we may assume that they form a basis for the rank $2$ subdiagram
spanned by $\alpha$ and $\beta$. By~\cite{AL} [Thm. 7] $\alpha,\beta$ can be
extended to a system of fundamental roots for $\Phi$. If in the associated
diagram there is a neighbour $\gamma$ of $\beta$ with $A_{\gamma\beta}\neq
\pm2$ and $\gamma$ is not a neighbour of $\alpha$, then $h_{\gamma}(r), r\in
R^{*}\setminus\{1\},$ is as required. We now deal with the remaining
situations separately either by finding a suitable $\gamma$ or by giving the
witness directly.

\medskip\noindent\textbf{$\Phi=E_{n},n=6,7,8$.} Since any pair of adjacent
fundmental roots is contained in an $A_{3}$ subdiagram, we may assume that
$\alpha=a_{2},\beta=a_{1}$ so $\gamma=a_{1}+a_{2}+a_{3}$ is as required.

\medskip\noindent\textbf{$\Phi=F_{4}$.} Let $a_{1},\ldots,a_{4}$ be the
resulting fundamental system where $a_{1}$ is long, $a_{n}$ is short. First
assume $\alpha=a_{2},\beta=a_{1}$, so $\gamma=a_{2}+2a_{3}$ is as required. If
$\alpha=a_{3},\beta=a_{4}$, then $\gamma=a_{2}+a_{3}$ is as required. If
$\alpha=a_{1},\beta=a_{2}$, and $\mathrm{char}(R)\neq2$, then $h_{\alpha}(-1)$
is as required. If $char(R)=2$, then $h_{a_{3}}(r)$ for $r\in R^{\ast
}\setminus\{1\}$ works.
\end{proof}

We can now summarize the existence of torus witnesses as follows:

\begin{proposition}
\label{prop:torus witness} Suppose $\Phi\in\{E_{6},E_{7},E_{8},F_{4}\}$ and
let $\alpha,\beta\in\Phi$ be linearly independent. Then there is a torus
witness for $(\alpha,\beta)$ \textbf{except} possibly if $R^{\ast}=\{\pm1\}$,
$\Phi=F_{4},$ and $\alpha,\beta$ are orthogonal and both long.
\end{proposition}

This completes the proof of Lemma \ref{mainlemma} for $\Phi\in\{E_{6}%
,E_{7},E_{8}\}.$

\bigskip

Assume finally that $G$ is of type $F_{4}$. Let $a_{1},\ldots,a_{4}$ be
fundamental roots of $\Phi$ where $a_{1},a_{2}$ are long, $a_{3},a_{4}$ are
short and $\alpha=a_{1}$. By Proposition~\ref{prop:torus witness} there is a
torus witnesses $s_{\alpha,\beta}$ for each root $\beta\neq\pm\alpha$, with
the exception of the following long roots:

\begin{enumerate}
\item $b_{2}= a_{1} + 2a_{2} +2a_{3},$

\item $b_{3}=a_{1}+2a_{2}+2a_{3}+2a_{4}$,

\item $b_{4}=a_{1}+2a_{2}+4a_{3}+2a_{4}$.
\end{enumerate}

Note that the root subgroups $U_{1},\ldots,~U_{4}$ corresponding to
$a_{1},b_{2},b_{3},b_{4}$ commute elementwise.

To complete the proof of Lemma \ref{mainlemma} for $G=F_{4}$ set
\[
Y=\left\{  s_{\alpha,\beta}~\mid\beta\in\Phi_{+}\smallsetminus\{~a_{1}%
,b_{2},b_{3},b_{4}\}\right\}  \cup\left\{  x_{-b_{i}}(1)~\mid~i=1,2,3\right\}
.
\]
Note that each $x_{-b_{i}}(1)$ centralizes each $U_{j}$ ($j\neq i$), and
commutes with no element of $U_{i}\smallsetminus\{1\}$.

Let $g\in\mathrm{C}_{G}(Y).$ We have to show that $g\in U_{1}Z$.

Arguing as in the proof of Proposition~\ref{prop:double centralizer} we
conclude that $g$ is of the form
\[
g=u_{1}u_{2}u_{3}u_{4}z\quad\mbox{ where }\quad u_{i}\in U_{i}%
,~\ i=1,2,3,4,~z\in Z.
\]

Since $x_{-b_{3}}(1)$ centralizes $g$ and $U_{1},U_{2},U_{4}$, we have
$u_{3}=1$. We see similarly that $u_{2}=u_{4}=1$, and the result follows.

\section{The building for $G_{2}$\label{building}}

Another way to study centralizers is to examine the action of $G=G(\overline
{k})$ on the building associated to $G$. This method is practical for groups
of rank $2$; we illustrate it here in the case of $G_{2}$, by proving

\begin{proposition}
\label{prop:G_2} Let $G$ be of type $G_{2}$ and let $U$ be a root group of
$G$. Then there exists a finite set $Y\subseteq\mathrm{C}_{G(R)}(U)$ such that
$\mathrm{C}_{G}(Y)\subseteq U$.
\end{proposition}

If $G$ is a Chevalley group of type $G_{2}$, we have $\mathrm{Z}(G)=1$ (see
\cite{St}, p. 23), and the associated spherical building~$\Delta$ is a
\emph{generalized hexagon}, i.e. a bipartite graph of diameter $6$, girth $12$
and valencies at least $3$ (see \cite{VM} for more details).

For vertices $x_{0},\ldots x_{m}$ in $\Delta,$
\[
G_{x_{0},\ldots,x_{m}}^{[i]}%
\]
denotes the subgroup of $G$ fixing all elements at distance at most $i$ from
some $x_{j}\in\{x_{0},\ldots,x_{m}\}$.

For $i=0$, this is just the pointwise stabilizer of $\{x_{0},\ldots,x_{m}\}$
in $G$ and we omit the superscript. In this notation, a \emph{root subgroup}
for a generalized hexagon $\Delta$ is of the form
\[
U=G_{x_{1},\ldots,x_{5}}^{[1]}%
\]
for a simple (i.e. without repetitions) path $(x_{1},\ldots,x_{5})$ in
$\Delta$. Thus our aim is to construct a finite set $Y\subseteq G(R)$
centralizing $G_{x_{1},\ldots,x_{5}}^{[1]}$ such that%
\[
g\in\mathrm{C}_{G}(Y)\text{ implies }g\in G_{x_{1},\ldots,x_{5}}^{[1]}.
\]

The generalized hexagon $\Delta$ associated to a Chevalley group of type
$G_{2}$ is a \emph{Moufang hexagon}, i.e. for any simple path $x_{0}%
,\ldots,x_{6}$ in $\Delta$ the root subgroup $G_{x_{1},\ldots,x_{5}}^{[1]}$
acts regularly on the set of neighbours of $x_{0}$ different from $x_{1}$ and
regularly on the set of neighbours of $x_{6}$ different from $x_{5}$ (see
\cite{TW}). As a consequence, we have
\begin{equation}
G_{x_{0},x_{1},\ldots,x_{5}}^{[1]}=1. \label{eq:Moufang}%
\end{equation}

We will repeatedly use the following:

\begin{remark}
\label{rem:parabolic} For any vertex $x\in\Delta$, the stabilizer $G_{x}$ is a
parabolic subgroup of $G$ and acts on the set of neighbours of $x$ as the
Zassenhaus group $\mathrm{PSL}_{2}(\overline{k})$. In particular, if $g\in G$
fixes at least three neighbours of $x$, then it fixes all neighbours of $x$.

Furthermore, for a path $(x,y)$ the stabilizer $G_{x,y}$ contains a regular
abelian normal subgroup acting as the additive group of $\overline{k}$ on the
set of neighbours of $y$ different from $x$.
\end{remark}

Most arguments rely on the following observation:

\begin{remark}
\label{rem:fix} Let $H$ be a group acting on a set $X$, let $g\in H$ and let
$A$ be the set of fixed points of $g$. Any $h\in H$ centralizing $g$ leaves
the set $A$ invariant.
\end{remark}

In light of (\ref{eq:Moufang}) this remark immediately implies

\begin{corollary}
\label{cor:center} For any root element $u\in G_{x_{1},\ldots,x_{5}}%
^{[1]}\setminus\{1\},$ each $g\in\mathrm{C}_{G}(u)$ fixes $x_{3}$.
\end{corollary}

For any $12$-cycle $(x_{0},\ldots,x_{12}=x_{0})$ in $\Delta$, the group
$G_{x_{0},\ldots,x_{12}}$ is a maximal torus in $G$. We let $U_{i}%
=G_{x_{i},\ldots,x_{i+4}}^{[1]}$, ($i=0,\ldots,11$) denote the corresponding
root subgroups (where addition is modulo $12$), so $U_{1}=U$.
In this notation we see that for $1\neq v\in U_{i}$ and $g\in\mathrm{C}%
_{G}(v)$ we have $g\in G_{x_{i+2}}$ by Corollary~\ref{cor:center}.

The bipartition of the vertices leads to two types of paths $(x_{0}%
,\ldots,x_{6})$ depending on the type of the initial vertex $x_{0}$ (note that
$x_{0}$ and $x_{6}$ have the same type). Since $G$ acts transitively on
ordered cycles of length $12$ (of the same bipartition type), the isomorphism
type of a root subgroup only depends on the type of the root group with
respect to this bipartition.

\medskip

\noindent It follows easily from the commutation relations that the root
subgroups corresponding to long roots consist of \emph{central elations}, i.e.
for one type of path $(x_{0},\ldots,x_{6})$ we have $G_{x_{1},\ldots,x_{5}%
}^{[1]}=G_{x_{3}}^{[3]}$.

First assume that $U=U_{1}=G_{x_{3}}^{[3]}$. \ Since $U$ centralizes $U_{j}$
for$~j=10,11,0,1,2,3,4,$ we may choose $Y$ to contain a nontrivial element
from each of the $U_{j}(R),~j=10,11,0,1,2,3,4$. We add five further elements
$y_{i}=v_{i}^{h_{i}}$ to $Y$ where%
\[
v_{1},v_{2} \in U_{3}(R);~v_{3}\in U_{4}(R),~v_{4}\in~U_{11}(R),~v_{5}\in
U_{10}(R);
\]
\[
h_{1} \in~U_{11}(R);~h_{2},h_{3}\in~U_{0}(R);~h_{4}\in U_{3}(R),~h_{5}\in
U_{2}(R),
\]
and $v_{i}\neq1,$ $h_{i}\neq1$ for each $i$. These centralize $U$ because for
$i=1,~2$ we have $[U,y_{i}]\subseteq G_{x_{5}^{h_{i}}}^{[3]}\cap G_{x_{3}%
}^{[3]}=1,$ $[U,y_{3}]\subseteq G_{x_{6}^{h_{3}}}^{[3]}\cap G_{x_{3}}%
^{[3]}=1,$ $[U,y_{4}]\subseteq G_{x_{1}^{h_{4}}}^{[3]}\cap G_{x_{3}}^{[3]}=1$
and $[U,y_{5}]\subseteq G_{x_{0}^{h_{5}}}^{[1]}\cap G_{x_{3}}^{[3]}=1$.

Now suppose that $g$ centralizes $Y$. Then $g\in G_{x_{0},\ldots,x_{6}}$. We
claim that $g\in G_{x_{i}}^{[1]}$ for $~i=1,\ldots,5$. Since $g$ commutes with
$y_{1},$ $g$ fixes $x_{5}^{h_{1}}\neq x_{3},x_{5}$. By
Remark~\ref{rem:parabolic} this implies that $g\in G_{x_{4}}^{[1]}$. Using
$y_{4}$ we see similarly that $g\in G_{x_{2}}^{[1]}$.

Similarly, since $g$ commutes with $y_{2},$ $g$ fixes $x_{5}^{h_{2}}$ and
hence also $x_{4}^{h_{2}}\neq x_{2},x_{4}$. This implies that $g\in G_{x_{3}%
}^{[1]}$. Finally, $g$ fixes $x_{6}^{h_{3}}$ and $x_{0}^{h_{5}}$\ because $g$
commutes with $y_{3}$ and $y_{5}$; as before we conclude from
Remark~\ref{rem:parabolic} that $g\in G_{x_{1},x_{5}}^{[1]}$, and the claim follows.

\medskip Now assume that $U=U_{1}\neq G_{x_{3}}^{[3]}$, and so $U_{2i}%
=G_{x_{2i+2}}^{[3]}$ for $i=0,\ldots,5$. Then $U$ commutes elementwise with
$U_{10},U_{0},U_{2},U_{4}$, and we choose $Y$ to contain a nontrivial element
from each of $U_{10}(R),~U_{0}(R),~U_{2}(R),~U_{4}(R).$

This ensures, by Remark~\ref{rem:fix}, that any element centralizing $Y$ must
lie in
\[
G_{x_{0},x_{2},x_{4},x_{6}}=G_{x_{0},x_{1},x_{2},x_{3},x_{4},x_{5},x_{6}}.
\]

As in the previous case, we extend $Y$ by four or six further elements
$y_{i}=v_{i}^{h_{i}},$ where $v_{i}\neq1,$ $h_{i}\neq1$ for each $i$,
\begin{align*}
v_{1}  &  \in U_{4}(R),~v_{2}\in U_{2}(R),~v_{3}\in U_{1}(R),~v_{4}\in
U_{10}(R);\\
h_{1}  &  \in~U_{0}(R),~h_{2}\in~U_{10}(R),~h_{3}\in U_{3}(R),~h_{4}\in
U_{2}(R),
\end{align*}
and if $\mathrm{char}(R)=3$ also
\[
v_{5}\in U_{10}(R),~v_{6}\in U_{4}(R);~h_{5}\in U_{3}(R),~h_{6}\in U_{11}(R).
\]
Note that $y_{1}$ and $y_{2}$ centralize $U$ because
\[
\lbrack U,y_{1}]\in G_{x_{6}^{h_{1}}}^{[3]}\cap G_{x_{2}}^{[1]}=1,~[U,y_{2}%
]\in G_{x_{4}^{h_{2}}}^{[3]}\cap G_{x_{0}}^{[1]}=1.
\]
Now let $g\in\mathrm{C}_{G}(Y).$ Then $g$ centralizes $y_{1},$ and therefore
fixes $x_{6}^{h}\neq x_{4},x_{6}$. By Remark~\ref{rem:parabolic} we get $g\in
G_{x_{5}}^{[1]}.$ In a similar way we find that $g\in G_{x_{1}}^{[1]}$ and
$g\in G_{x_{3}}^{[1]}$.

It remains to show that $g\in G_{x_{i}}^{[1]}$ for $i=2$ and $i=4$. We
distinguish two cases according to the characteristic of $R$. First assume
that $\mathrm{char}(R)\neq3$ and extend the path $(x_{1},\ldots,x_{5})$ to a
simple path $(x_{1},\ldots,x_{7})$. For any $v\in G_{x_{3},\ldots,x_{7}}%
^{[1]}\setminus\{1\}$ \ and $1\neq u\in U_{1}$ the commutator relations (see
\S \ref{comfor}) with $\mathrm{char}(k)\neq3$ imply that $[u,v]\neq1$. This
shows that $x_{1},x_{3}$ are the only neighbours $y$ of $x_{2}$ such that
$G_{y}^{[1]}$ meets $U_{1}$ nontrivially.

On the other hand, for any simple path $(x_{1}^{\prime},x_{2},x_{3}%
,x_{4},x_{5})$ the actions of the root groups $U_{1}$ and $U_{1}^{\prime
}=G_{x_{1}^{\prime},x_{2},x_{3},x_{4},x_{5}}^{[1]}$ on the neighbours of
$x_{6}$ agree, by Remark~\ref{rem:parabolic}. Since the root groups are
abelian, we therefore have $[U,w]=1$ for any $w\in U_{1}^{\prime}$.

This shows in particular that $y_{3}\in\mathrm{C}_{G(R)}(U)$. By the previous
remark $x_{1}^{\prime}=x_{1}^{h_{3}}$ and $x_{3}$ are the only neighbours of
$x_{2}$ such that $y_{3}\in G_{x_{1}^{\prime},x_{3}}^{[1]}$, and so $g$ fixes
$x_{1}^{\prime}$. Again by Remark~\ref{rem:parabolic} we conclude that $g\in
G_{x_{2}}^{[1]}$. \ Similarly, we see that $y_{3}\in\mathrm{C}_{G(R)}(U)$, and
find that $g\in G_{x_{4}}^{[1]}$ as required.

Finally assume that $\mathrm{char}(R)=3$. Then the commutation relations show
that $[U_{1},U_{3}]=1$ and hence we have
\[
U_{1}=G_{x_{2},x_{4}}^{[2]}.
\]
As $h_{5}\in U_{3}$ and $v_{5}\in U_{10}=G_{x_{0}}^{[3]}$, we have
$[U,y_{5}]\in U_{1}\cap G_{x_{0}^{h}}^{[3]}=1$, so $y_{5}\in\mathrm{C}%
_{G(R)}(U)$. Now Corollary~\ref{cor:center} implies that $g$ fixes
$x_{0}^{h_{5}}$ and hence also $x_{1}^{h_{5}}\neq x_{1},x_{3}$. As before we
infer that $g\in G_{x_{2}}^{[1]}$. The same argument using $y_{6}$ shows
finally that $g\in G_{x_{4}}^{[1]}$, and concludes the proof.

\section{Root witnesses in the classical groups\label{class}}

In this section we establish Lemma \ref{mainlemma} for the groups of classical
type, and complete the proof of Theorem \ref{thm:double centralizer}.

\begin{proposition}
\label{dct}Let $G$ be a Chevalley group of type $A_{n},$ $B_{m},~C_{m}$ or
$D_{m}$ ($n\geq3, \ m\geq2$), and let $R$ be an integral domain. Let $U$ be a
root subgroup of $G$. Write $Z$ for the centre of $G$. There exists a set
$Y\subseteq\mathrm{C}_{G(R)}(U)$ consisting of root elements such that%
\begin{equation}
\mathrm{C}_{G}(Y)\subseteq UZ, \label{rootwit}%
\end{equation}
unless $G$ is of type $C_{n},$ $U$ belongs to a short root $\alpha$ and
$R^{\ast}=\{\pm1\}$, in which case%
\begin{equation}
\mathrm{C}_{G}(Y)\subseteq UU_{1}U_{2}Z \label{ii}%
\end{equation}
where $U_{1}$ and $U_{2}$ are root subgroups belonging to long roots adjacent
to $\alpha$ in a $C_{2}$ subsystem.
\end{proposition}

\begin{proof}
Suitable sets $Y$ are exhibited in the lemmas below for particular forms of
$G:$ the universal groups $\mathrm{SL}_{n}$ and $\mathrm{Sp}_{2m}$ for
$A_{n},$ $C_{m}$ respectively, and for orthogonal versions of $B_{m}$ and
$D_{m}$. Now if $v$ is a unipotent element and $v^{g}\in vZ$ then $v^{g}=v,$
because $Z$ consists of semisimple elements (Jordan decomposition); hence both
statements involving $Y$ remain true if $\mathrm{C}_{G(R)}$ is replaced by
`centralizer modulo $Z$'. It follows easily that if (\ref{rootwit}) or
(\ref{ii}) holds, then it remains valid when $G$ is replaced by $G/Z.$ In
particular, they hold for the adjoint form of each group, and any group
`between' $\mathrm{SL}_{n}$ and $\mathrm{PSL}_{n}$.

The result for the universal forms (in cases $B_{m}$ and $D_{m}$) follows
directly from the established cases because root elements in $G(R)$ lift to
root elements in the covering group.
\end{proof}

\medskip

The precise description of $\mathrm{Z}(\mathrm{C}_{G(R)}(u))$ for $1\neq u\in
U$ in Case (\ref{ii}) is given below in Proposition \ref{exactdc}.

We use the notation of \cite{C}, \S 11.3 for the classical groups. Throughout,
$R$ denotes an integral domain, and $e_{ij}$ the matrix with one non-zero
entry equal to $1$ in the $(i,j)$ place. We call a set $Y\subseteq
\mathrm{C}_{G(R)}(U_{\alpha})$ satisfying (\ref{rootwit}), resp. (\ref{ii}) a
\emph{witness set} for $U_{\alpha}$.

In most cases, the verification that $Y$ has the required properties is a
relatively straightforward matrix calculation, which we omit. Of course it
will suffice to consider just one root of each length.

\subsection*{The special linear group}

The root subgroups in $\mathrm{SL}_{n}$ are%
\[
U_{ij}=1+\overline{k}e_{ij},~\ i\neq j.
\]

\begin{lemma}
Let $G=\mathrm{SL}_{n}$, $n\geq2$. Then a witness set for $U_{12}$ is
\[
Y=\{1+e_{pq}~\mid~p\neq2,~q\neq1\}.
\]

\end{lemma}

\subsection*{Symplectic groups and even orthogonal groups}

Now we consider $C_{m}(\overline{k})$ and $D_{m}(\overline{k})$ as groups of
$2m\times2m$ matrices, as described in \cite{C}, \S 11.3. Here $n=2m$ and we
re-label the matrix entries writing $-i$ in place of $m+i,$ ($i=1,\ldots,m$).
For $1\leq\left\vert i\right\vert <\left\vert j\right\vert \leq m$ set%
\begin{equation}
\alpha_{ij}=e_{ij}+\varepsilon e_{-j,-i}, \label{alpha}%
\end{equation}
where $\varepsilon=\pm1$ depends on $(i,j)$ in a manner to be specified.

\bigskip We now separate cases.\medskip

\textbf{Case 1:} $G=C_{m}=\mathrm{Sp}_{2m}.$ \ In this case, $\varepsilon$ is
$-1$ or $1$ according as $i$ and $j$ have the same or opposite signs. The root
subgroups in $G$ are%
\begin{align*}
U_{i}  &  =1+\overline{k}e_{i,-i}\text{ \ \ (long roots), \ }1\leq\left\vert
i\right\vert \leq m\\
U_{ij}  &  =1+\overline{k}\alpha_{ij}\text{ \ (short roots), \ }%
~1\leq\left\vert i\right\vert <\left\vert j\right\vert \leq m,
\end{align*}
taking $\varepsilon=-1$ if $ij>0,$ $\varepsilon=1$ if $ij<0.$

\medskip

\begin{lemma}
\label{symp}Let $G=\mathrm{Sp}_{2m}$. A witness set for the long root group
$U_{1}$ is%
\[
X_{1}=\{1+e_{i,-i}~\mid~i\notin\{-1,2\}~\}\cup\{1+\alpha_{1j}~\mid2\leq~j\leq
m~\}
\]
and a witness set for the short root group $U_{12}$ is%
\[
X_{2}=\left\{  1+e_{i,-i}~\mid i\neq-1,2\right\}  \cup\{1+\alpha_{1j}~\mid
j\neq\pm1,-2~\}~.
\]

\end{lemma}

Now let $v=1+r\alpha_{12}\in U_{12}$, $0\neq r\in R$. To identify the subgroup
$\mathrm{Z}(\mathrm{C}_{G(R)}(v))$ more precisely, set%
\[
\xi=(e_{1,-2}-e_{2,-1})-(e_{-1,2}-e_{-2,1})+\sum_{\left\vert i\right\vert
>2}e_{ii}.
\]
Then $\xi\in\mathrm{C}_{G(R)}(v)$. If $g\in ZUU_{1}U_{2}$ and $g$ commutes
with $\xi$ we find that%
\begin{align}
g  &  =\pm(1+c\alpha_{12})(1+ae_{1,-1})(1-ae_{-2,2})\label{gandphi}\\
&  =\pm(1+c\alpha_{12})\cdot\phi(a)\nonumber
\end{align}
for some $a,~c\in\overline{k}$, where $\phi:\overline{k}\rightarrow
U_{1}U_{-2}$ is the `diagonal' homomorphism
\[
r\longmapsto1+r(e_{1,-1}-e_{-2,2})=(1+re_{1,-1})(1-re_{-2,2}).
\]
Now we can state

\begin{proposition}
\label{exactdc}Let $G$ and $v$ be as above. Then%
\begin{align}
\mathrm{Z}(\mathrm{C}_{G(R)}(v))  &  \leq\pm U_{12}(R)\cdot\phi(R),
\label{eq1}\\
\mathrm{Z}(\mathrm{C}_{G(R)}(v))  &  =\pm U_{12}(R)\text{ ~~~~~~~\ if }%
R^{\ast}\neq\{\pm1\},\label{mid}\\
\mathrm{Z}(\mathrm{C}_{G(R)}(v))  &  =\pm U_{12}(R)\cdot\phi(R)\text{ if
}R^{\ast}=\{\pm1\}\text{ and }\mathrm{char}(R)\neq2. \label{eq2}%
\end{align}

\end{proposition}

\begin{proof}
We have already established that $\mathrm{C}_{G}(\mathrm{C}_{G(R)}(v))\leq\pm
U_{12}\cdot\phi(\overline{k}).$ If $g$ is given by (\ref{gandphi}), both $c$
and $a$ appear as entries in the matrix $g$, so if $g\in G(R)$ then $a,c\in R$
and (\ref{eq1}) follows.

Suppose now that $R^{\ast}\neq\{\pm1\}$ and pick $t\in R^{\ast}$ with
$t^{2}\neq1$. The torus element%
\[
\tau:=h_{1,-2}(t)=t(e_{11}+e_{22})+t^{-1}(e_{-2,-2}+e_{-1,-1})
\]
lies in $\mathrm{C}_{G(R)}(v)$. So if $g$ in (\ref{gandphi}) is in
$\mathrm{Z}(\mathrm{C}_{G(R)}(v))$ then $\tau$ commutes with $\phi(a)$, and
hence with $\phi(a)-1=r(e_{1,-1}-e_{-2,2})$. But
\[
\tau^{-1}\cdot r(e_{1,-1}-e_{-2,2})\cdot\tau=t^{-2}re_{1,-1}-t^{2}re_{-2,2},
\]
so $t^{-2}r=t^{2}r=r,$ $r=0$ and we conclude that $g\in\pm U_{12}(R)$. This
proves (\ref{mid}).

Assume now that $R^{\ast}=\{\pm1\}$ and $\mathrm{char}(R)\neq2$. To establish
(\ref{eq2}) it will suffice to show that $e_{1,-1}-e_{-2,2}$ commutes with
every matrix in $\mathrm{C}_{G(R)}(v).$

For clarity we take $n=3$; the argument is valid for any $n\geq2$. A matrix
commuting with $v$ is of the form%
\[
g=\left[
\begin{array}
[c]{cccccc}%
x & \bullet & \bullet & \bullet & -b & \bullet\\
0 & x & 0 & b & 0 & 0\\
0 & \bullet & \bullet & \bullet & 0 & \bullet\\
0 & a & 0 & y & 0 & 0\\
-a & \bullet & \bullet & \bullet & y & \bullet\\
0 & \bullet & \bullet & \bullet & 0 & \bullet
\end{array}
\right]  ,
\]
where the blank entries are arbitrary. If $g$ is symplectic then%
\begin{align*}
2ax  &  =2by=0\\
xy+ab  &  =1.
\end{align*}
It follows that \emph{either }$x=0$, in which case $ab=1,$ whence $a=\pm1=b$
and $y=0$, \emph{or} $x\neq0$, in which case $a=0$, $xy=1$ and similarly then
$x=\pm1=y$ and $b=0$. Thus in any case $x=y$ and $a=b$. This now implies that
$g$ commutes with $e_{1,-1}-e_{-2,2}$.
\end{proof}

\medskip

\textbf{Remark}\emph{.} The precise nature of $\mathrm{Z}(\mathrm{C}%
_{G(R)}(v))$ in the remaining case where $R^{\ast}=1$ and $\mathrm{char}(R)=2$
we leave open.

\bigskip

\textbf{Case 2:} $G=D_{m}\leq\mathrm{O}_{2m}.$ \ In this case, $\varepsilon
=-1$ for all $i,j$. The root subgroups in $G$ are%
\[
U_{ij}=1+\overline{k}\alpha_{ij},~1\leq\left\vert i\right\vert <\left\vert
j\right\vert \leq m.
\]

\begin{lemma}
\label{D-dc}Let $G=D_{m}\leq\mathrm{O}_{2m}$. A witness set for the root group
$U_{12}$ is%
\begin{equation}
X_{3}=\left\{  1+\alpha_{ij}~\mid(i,j)\in S~\right\}  \label{Xthree}%
\end{equation}
where%
\begin{equation}
S=\left\{  (i,j)~\mid~3\leq\left\vert i\right\vert <\left\vert j\right\vert
\text{ or }i=1<\left\vert j\right\vert \right\}  \cup\{(-1,2)\}. \label{Slist}%
\end{equation}

\end{lemma}

\textbf{Remark} The same calculation actually establishes a little more:
namely,%
\begin{equation}
\mathrm{C}_{\mathrm{O}_{2m}}(X_{3})\subseteq\pm U_{12}. \label{sharp}%
\end{equation}
\ This will be used below.

\subsection*{Odd orthogonal groups}

Now we take $G=B_{m}\leq\mathrm{O}_{2m+1},$ and write elements of $G$ as
matrices%
\[
g=\left(
\begin{array}
[c]{cc}%
x & a\\
b^{T} & h
\end{array}
\right)  :=(x,a,b;h)
\]
where $x=x(g)\in\overline{k}$, $a=a(g)$ and $b=b(g)$ are in $\overline{k}%
^{2m}$ and $h=h(g)\in M_{2m}(\overline{k})$. For $h\in M_{2m}(\overline{k})$
we write%
\[
h^{\ast}=(1,0,0;h).
\]
The rows and columns are labelled $0,1,\ldots,m,-1,\ldots,-m.$

We begin with a couple of elementary observations.

\begin{lemma}
\label{BvsD}Let $g=(x,0,0;h)$. Then $g\in\mathrm{O}_{2m+1}(\overline{k})$ if
and only if $h\in\mathrm{O}_{2m}(\overline{k})$ and $x=\pm1$.
\end{lemma}

\begin{lemma}
\label{starcom}Let $w\in M_{2m}(\overline{k})$. Then $g=g(x,a,b;h)$ commutes
with $w^{\ast}$ if and only if%
\begin{align*}
hw  &  =wh\\
aw  &  =a,~bw^{T}=b.
\end{align*}

\end{lemma}

The root elements are%
\begin{align*}
u_{i}(r)  &  =1+r(2e_{i0}-e_{0,-i})-r^{2}e_{i,-i}\text{ \ \ (short roots),
\ }1\leq\left\vert i\right\vert \leq m\\
u_{ij}(r)  &  =1+r\alpha_{ij}\text{ \ (long roots), \ }~1\leq\left\vert
i\right\vert <\left\vert j\right\vert \leq m,
\end{align*}
where $\alpha_{ij}$ are as in (\ref{alpha}) with $\varepsilon=-1$ for all
pairs $i,j$.

Now let $r\neq0$ and consider the long root element $v^{\ast}=u_{12}(r)$ (so
$v$ is the corresponding root element in $D_{m}$). We have%
\[
\mathrm{C}_{G}(v^{\ast})\supseteq X_{3}^{\ast}%
\]
where $X_{3}$ is defined above (\ref{Xthree}). Now Lemma \ref{starcom}
implies: if $g=g(x,a,b;h)\in\mathrm{Z}(\mathrm{C}_{G}(v^{\ast}))$ then
$a\alpha_{ij}$ and $b\alpha_{ji}$ are zero for all pairs $(i,j)\in S$ (see
(\ref{Slist})). This now implies that $a=b=0$.

It follows by Lemma \ref{BvsD} that $x=\pm1$ and $h\in\mathrm{O}%
_{2m}(\overline{k}),$ and then by (\ref{sharp}) that
\[
h\in\pm U_{12}%
\]
(here $U_{12}$ is the corresponding root group in $D_{m}$).

Thus
\[
g=(\pm1,0,0;\pm(1+s\alpha_{12}))=\pm u_{12}(s)\cdot(\eta,1,\ldots1)
\]
for some $s\in\overline{k}$ and $\eta=\pm1$.

Finally, we note that $u_{1}(1)\in\mathrm{C}_{G}(v^{\ast})$. It follows that
$(\eta,1,\ldots1)$ commutes with $u_{1}(1),$ which forces $\eta=1.$ Thus
$g=\pm u_{12}(s).$ We have established

\begin{lemma}
Let $G=B_{m}\leq\mathrm{O}_{2m+1}$. A witness set for the long root group
$U_{12}$ is
\[
X_{4}=X_{3}^{\ast}\cup\{u_{1}(1)\}.
\]

\end{lemma}

Assume henceforth that $m\geq3$. We consider finally the short root group
$U_{1}$. We see that $\mathrm{C}_{G}(U_{1})$ contains the set%
\[
X_{5}=\left\{  u_{ij}(1)\mid~i\neq-1,~j\neq1\right\}  \cup\{u_{1}(1)\}.
\]
Now let $g=g(x,a,b;h)\in\mathrm{C}_{G}(X_{5})$. One finds after some
calculation that%
\[
g=(x,se_{-1},2e_{1};x\mathbf{1}_{2m}+ye_{1,-1}).
\]
(This calculation requires $m\geq3$; the conclusion is false when $m=2$).

Then $\det(g)=x^{2m+1}$ so $x$ is invertible; replacing $s$ by $-x^{-1}s$ and
$y$ by $x^{-1}y$ we have
\begin{equation}
g=x(1,-se_{-1},2se_{1};\mathbf{1}_{2m}+ye_{1,-1}). \label{specialg}%
\end{equation}
Then%
\[
g\cdot u_{1}(-s)=x(\mathbf{1}+(s^{2}+y)e_{1,-1})\in\mathrm{O}_{2m+1},
\]
which implies $x^{2}=1$ and $2(s^{2}+y)=0$.

If $\mathrm{char}(k)\neq2$ we infer that $g=\pm u_{1}(s).$

Suppose now that $k$ has characteristic $2$. In this case the mapping
$\pi:g\longmapsto h(g)$ is an injective homomorphism (\cite{C}, page 187). If
$g$ is of the form (\ref{specialg}) and $y=w^{2}$ then $g\pi=u_{1}(w)\pi\in
U_{1}\pi$, and so $g\in U_{1}$.

Thus is any case we have $g\in\pm U_{1}$. We have established

\begin{lemma}
Let $G=B_{m}\leq\mathrm{O}_{2m+1},$ where $m\geq3$. Then a witness set for the
short root group $U_{12}$ is%
\[
X_{5}=\left\{  u_{ij}(1)\mid~i\neq-1,~j\neq1\right\}  \cup\{u_{1}(1)\}.
\]

\end{lemma}

\section{Adelic groups}

Let $\mathbb{A}$ denote the ad\`{e}le ring of a global field $K$, with
$\mathrm{char}(K)\neq2,3,5$. We consider subrings of $\mathbb{A}$ of the
following kind:%
\[
A=\mathbb{A},~~A=~\prod_{\mathfrak{p}\in\mathcal{P}}\mathfrak{o}%
_{\mathfrak{p}}%
\]
where $\mathfrak{o}$ is the ring of integers of $K$ and $\mathcal{P}$ is a
non-empty set of primes (or places) of $K$. Here we establish

\begin{theorem}
\label{sl2}The ring $A$ is bi-interpretable with each of the groups
$\mathrm{SL}_{2}(A)$, $\mathrm{SL}_{2}(A)/\left\langle -1\right\rangle ,$
$\mathrm{PSL}_{2}(A)$.
\end{theorem}

\begin{theorem}
\label{high_rank}Let $G$ be a simple Chevalley-Demazure group scheme of rank
at least $2$. Then $A$ is bi-interpretable with the group $G(A)$.
\end{theorem}

For a rational prime $p$ we write $A_{p}=\prod_{\mathfrak{p}\in\mathcal{P}%
,~\mathfrak{p}\mid p}\mathfrak{o}_{\mathfrak{p}}$.

\begin{lemma}
\label{S-lemma}$A$ has a finite subset $S$ such that every element of $A$ is
equal to one of the form%
\begin{equation}
\xi^{2}-\eta^{2}+s \label{formula}%
\end{equation}
with $\xi,\eta\in A^{\ast}$ and $s\in S$.
\end{lemma}

\begin{proof}
In any field of characteristic not $2$ and size $>5$, every element is the
difference of two non-zero squares. It follows that the same is true for each
of the rings $\mathfrak{o}_{\mathfrak{p}}$ with $N(\mathfrak{p})>5$ and odd.

If $N(\mathfrak{p})$ is $3$ or $5$ then every element of $\mathfrak{o}%
_{\mathfrak{p}}$ is of the form (\ref{formula}) with $\xi,\eta\in
\mathfrak{o}_{\mathfrak{p}}^{\ast}$ and $s\in\{0,\pm1\}.$ If $\mathfrak{p}$
divides $2$, the same holds if $S$ is a set of representatives for the cosets
of $4\mathfrak{p}$ in $\mathfrak{o}$.

Now by the Chinese Remainder Theorem (and Hensel's lemma) we can pick a finite
subset $S_{1}$ of $A_{2}\times A_{3}\times A_{5}$ such that every element of
$A_{2}\times A_{3}\times A_{5}$ is of the form (\ref{formula}) with $\xi
,\eta\in\mathfrak{o}_{\mathfrak{p}}^{\ast}$ and $s\in S_{1}$. Finally, let $S$
be the subset of elements $s\in A$ that project into $S_{1}$ and have
$\mathfrak{o}_{\mathfrak{p}}$-component $1$ for all $\mathfrak{p}\nmid30$
(including infinite places if present).
\end{proof}

\medskip

\textbf{Remark} If $K=\mathbb{Q}$ one could choose $S\subset\mathbb{Z}$
(diagonally embedded in $A$). The plethora of parameters in the following
argument can then be replaced by just three - $h(\tau),~u(1),$ $v(1)$ - or
even two when $A=\mathbb{A}$, in which case we replace $h(\tau)$ by $h(2),$
which can be expressed in terms of $u(1)$ and $v(1)$ by the formula
(\ref{h-identity}) below. Also the formula (\ref{P-formula}) can be replaced
by the simpler one: $y_{2}=u^{x}u^{-y}u^{s}\wedge y_{3}=y_{1}^{x}y_{1}%
^{-y}y_{1}^{s}$.

\medskip

For a finite subset $T$ of $\mathbb{Z}$ let
\[
A_{T}=\left\{  r\in A~\mid~r_{\mathfrak{p}}\in T\text{ for every }%
\mathfrak{p}\right\}  .
\]
This is a definable set, since $r\in A_{T}$ if and only if $f(r)=0$ where
$f(X)=\prod_{t\in T}(X-t)$.

\medskip

Choose $S$ as in Lemma \ref{S-lemma}, with $0,~1\in S$, and write
$S^{2}=S\cdot S$.

\medskip

Let $\Gamma=\mathrm{SL}_{2}(A)/Z$ where $Z$ is $1$, $\left\langle
-1\right\rangle $ or the centre of $\mathrm{SL}_{2}(A)$. \ For $\lambda\in A$
write%
\[
u(\lambda)=\left(
\begin{array}
[c]{cc}%
1 & \lambda\\
0 & 1
\end{array}
\right)  ,~v(\lambda)=\left(
\begin{array}
[c]{cc}%
1 & 0\\
-\lambda & 1
\end{array}
\right)  ,~~h(\lambda)=\left(
\begin{array}
[c]{cc}%
\lambda^{-1} & 0\\
0 & \lambda
\end{array}
\right)  ~~(\lambda\in A^{\ast})
\]
(matrices interpreted modulo $Z$; note that $\lambda\longmapsto u(\lambda)$ is
bijective for each choice of $Z$).

Fix $\tau\in A^{\ast}$ with $\tau_{\mathfrak{p}}=2$ for $\mathfrak{p}\nmid2,$
$\tau_{\mathfrak{p}}=3$ for $\mathfrak{p}\mid2$. It is easy to verify that%
\begin{equation}
\mathrm{C}_{\Gamma}(h(\tau))=h(A^{\ast}):=H. \label{Hdef}%
\end{equation}

\begin{proposition}
\label{P1}The ring $A$ is definable in $\Gamma.$
\end{proposition}

\begin{proof}
We take $h:=h(\tau)$ and $\{u(c)~\mid~c\in S^{2}\}$ as parameters, and put
$u:=u(1)$. `Definable' will mean definable with these parameters. For
$\lambda\in A$ and $\mu\in A^{\ast}$ we have%
\[
u(\lambda)^{h(\mu)}=u(\lambda\mu^{2}).
\]

Now (\ref{Hdef}) shows that $H$ is definable. If $\lambda=\xi^{2}-\eta^{2}+s$
and $x=h(\xi),~y=h(\eta)$ then $u(\lambda)=u^{x}u^{-y}u(s)$; thus%
\[
U:=u(A)=\bigcup_{s\in S}\{u^{x}u^{-y}u(s)~\mid~x,~y\in H\}
\]
is definable.

The map $u:A\rightarrow U$ is an isomorphism from $(A,+)$ to $U$. It becomes a
ring isomorphism with multiplication $\ast$ if one defines%
\begin{equation}
u(\beta)\ast u(\alpha)=u(\beta\alpha). \label{thing}%
\end{equation}
We need to provide an $L_{\mathrm{gp}}$ formula $P$ such that for
$y_{1},~y_{2},~y_{3}\in U$,%
\begin{equation}
y_{1}\ast y_{2}=y_{3}\Longleftrightarrow\Gamma\models P(y_{1},y_{2},y_{3}).
\label{nextthing}%
\end{equation}

Say $\alpha=\xi^{2}-\eta^{2}+s$, $\beta=\zeta^{2}-\rho^{2}+t$. Then%
\[
u(\beta\alpha)=u(\beta)^{x}u(\beta)^{-y}u(s)^{z}u(s)^{-r}u(st)
\]
where $x=h(\xi),~y=h(\eta),~z=h(\zeta)~$and $r=h(\rho)$.

So we can take $P(y_{1},y_{2},y_{3})$ to be a formula expressing the
statement: there exist $x,~y,z,r\in H$ such that for some $s,t\in S$%
\begin{align}
y_{1}  &  =u^{z}u^{-r}u(t),~y_{2}=u^{x}u^{-y}u(s),\label{P-formula}\\
y_{3}  &  =y_{1}^{x}y_{1}^{-y}u(s)^{z}u(s)^{-r}u(st).\nonumber
\end{align}

\end{proof}

\begin{proposition}
The group $\Gamma$ is interpretable in $A$.
\end{proposition}

\begin{proof}
When $\Gamma=\mathrm{SL}_{2}(A)$, \ clearly $\Gamma$ is definable as the set
of $2\times2$ matrices with determinant $1$ and group operation matrix
multiplication. For the other cases, it suffices to note that the equivalence
relation `modulo $Z$' is definable by $B\thicksim C$ iff there exists
$Z\in\{\pm1_{2}\}$ with $C=BZ$, resp. $Z\in H$ with $Z^{2}=1$ and $C=BZ$.
\end{proof}

\bigskip

To complete the proof of Theorem \ref{sl2} it remains to establish
\textbf{Step 1 }and \textbf{Step 2 }below.

We take $v=v(1)$ as another parameter, and set $w=uvu=\left(
\begin{array}
[c]{cc}%
0 & 1\\
-1 & 0
\end{array}
\right)  .$ Then $u(\lambda)^{w}=v(\lambda)$, so $V:=v(A)=U^{w}$ is definable.
Note the identity (for $\xi\in A^{\ast}$):%
\begin{equation}
h(\xi)=v(\xi)u(\xi^{-1})v(\xi)w^{-1}=w^{-1}u(\xi)w\cdot u(\xi^{-1})\cdot
w^{-1}u(\xi). \label{h-identity}%
\end{equation}

\noindent\textbf{Step 1:} The ring isomorphism from $A$ to $U\subset
\mathrm{M}_{2}(A)$ is definable. Indeed, this is just the mapping%
\[
r\longmapsto\left(
\begin{array}
[c]{cc}%
1 & r\\
0 & 1
\end{array}
\right)  .
\]

\noindent\textbf{Step 2:} The map $\theta$ sending $g=(a,b;c,d)$ to
$(u(a),u(b);u(c),u(d))\in\Gamma^{4}$ is definable; this is a group isomorphism
when $U$ is identified with $A$ via $u(\lambda)\longmapsto\lambda$.

Assume for simplicity that $\Gamma=\mathrm{SL}_{2}(A)$. We start by showing
that the restriction of $\theta$ to each of the subgroups $U,~V\,,~H$ is
definable. Recall that $u(0)=1$ and $u(1)=u$.

If $g\in U$ then $g\theta=(u,g;1,u)$. If $g=v(-\lambda)\in V$ then
$g^{-w}=u(\lambda)\in U$ and $g\theta=(u,1;g^{-w},u).$

Suppose $g=h(\xi)\in H$. Then $g=w^{-1}xwyw^{-1}x$ where $x=u(\xi),$
$y=u(\xi^{-1}),$ and $g\theta=(y,1;1,x)$. So $g\theta=(y_{1},y_{2};y_{3}%
,y_{4})$ if and only if
\begin{align*}
~y_{4}\ast y_{1}  &  =u,~y_{2}=y_{3}=1,\\
g  &  =w^{-1}y_{4}wy_{1}w^{-1}y_{4}.
\end{align*}
Thus the restriction of $\theta$ to $H$ is definable.

Next, set%
\[
W:=\left\{  x~\in\Gamma\mid x_{\mathfrak{p}}\in\{1,w\}\text{ for every
}\mathfrak{p}~\right\}  .
\]
To see that $W$ is definable, observe that an element $x$ is in $W$ if and
only if there exist $y,z\in u(A_{\{0,1\}})$ such that
\[
x=yz^{w}y\text{ and }x^{4}=1.
\]
Note that $u(A_{\{0,1\}})$ is definable by (the proof of) Proposition \ref{P1}.

Put%
\[
\Gamma_{1}=\{g\in\Gamma\mid g_{11}\in A^{\ast}\}.
\]

If $g=(a,b;c,d)\in\Gamma_{1}$ then $g=\widetilde{v}(g)\widetilde{h}%
(g)\widetilde{u}(g)$ where%
\begin{align*}
\widetilde{v}(g)  &  =v(-a^{-1}c)\in V\\
\widetilde{h}(g)  &  =h(a^{-1})\in H\\
\widetilde{u}(g)  &  =u(a^{-1}b)\in U.
\end{align*}
This calculation shows that in fact $\Gamma_{1}=VHU$, so $\Gamma_{1}$ is
definable; these three functions on $\Gamma_{1}$ are definable since%
\begin{align*}
x  &  =\widetilde{v}(g)\Longleftrightarrow x\in V\cap HUg\\
y  &  =\widetilde{u}(g)\Longleftrightarrow y\in U\cap HVg\\
z  &  =\widetilde{h}(g)\Longleftrightarrow z\in H\cap VgU.
\end{align*}

Let $g=(a,b;c,d)$. Then $gw=(-b,a;-d,c)$. We claim that there exists $x\in W$
such that $gx\in\Gamma_{1}$. Indeed, this may be constructed as follows: If
$a_{\mathfrak{p}}\in\mathfrak{o}_{\mathfrak{p}}^{\ast}$ take $x_{\mathfrak{p}%
}=1$. If $a_{\mathfrak{p}}\in\mathfrak{po}_{\mathfrak{p}}$ and
$b_{\mathfrak{p}}\in\mathfrak{o}_{\mathfrak{p}}^{\ast}$ take $x_{\mathfrak{p}%
}=w$. If both fail, take $x_{\mathfrak{p}}=1$ when $a_{\mathfrak{p}}\neq0$ and
$x_{\mathfrak{p}}=w$ when $a_{\mathfrak{p}}=0$ and $b_{\mathfrak{p}}\neq0$.
This covers all possibilities since for almost all $\mathfrak{p}$ at least one
of $a_{\mathfrak{p}}$, $b_{\mathfrak{p}}$ is a unit in $\mathfrak{o}%
_{\mathfrak{p}},$ and $a_{\mathfrak{p}}$, $b_{\mathfrak{p}}$ are never both zero.

As $gx\in\Gamma_{1},$ we may write%
\[
gx=\widetilde{v}(gx)\widetilde{h}(gx)\widetilde{u}(gx)\text{.}%
\]

We claim that the restriction of $\theta$ to $W$ is definable. Let $x\in W$
and put $P=\{\mathfrak{p}~\mid~x_{\mathfrak{p}}=1\},$ $Q=\{\mathfrak{p}%
~\mid~x_{\mathfrak{p}}=w\}$. Then $(u^{x})_{\mathfrak{p}}$ is $u$ for
$\mathfrak{p}\in P$ and $v$ for $\mathfrak{p}\in Q$, so $u^{x}\in\Gamma_{1}$
and
\[
\widetilde{u}(u^{x})_{\mathfrak{p}}=\left\{
\begin{array}
[c]{ccc}%
u &  & (\mathfrak{p}\in P)\\
1 &  & (\mathfrak{p}\in Q)
\end{array}
\right.  .
\]
Recalling that $u=u(1)$ and $1=u(0)$ we see that%
\[
x\theta=\left(
\begin{array}
[c]{cc}%
\widetilde{u}(u^{x}) & \widetilde{u}(u^{x})^{-1}u\\
u^{-1}\widetilde{u}(u^{x}) & \widetilde{u}(u^{x})
\end{array}
\right)  .
\]

We can now deduce that $\theta$ is definable. Indeed, $g\theta=A$ holds if and
only if there exists $x\in W$ such that $gx\in\Gamma_{1}$ and
\[
A\cdot x\theta=\widetilde{v}(gx)\theta\cdot\widetilde{h}(gx)\theta
\cdot\widetilde{u}(gx)\theta
\]
(of course the products here are matrix products, definable in the language of
$\Gamma$ in view of Proposition \ref{P1}).

This completes the proof of Theorem \ref{sl2}\ for $\Gamma=\mathrm{SL}_{2}%
(A)$. When $\Gamma=\mathrm{SL}_{2}(A)/Z$, the same formulae now define
$\theta$ as a map from $\Gamma$ into the set of $2\times2$ matrices with
entries in $U$ modulo the appropriate definable equivalence relation.
$\blacksquare$

\bigskip

Now we turn to the proof of Theorem \ref{high_rank}. This largely follows
\S \ref{bisection}, but is simpler because we are dealing here with `nice'
rings. Henceforth $G$ denotes a simple Chevalley-Demazure group scheme of rank
at least $2$. The root subgroup associated to a root $\alpha$ is denoted
$U_{\alpha}$, and $Z$ denotes the centre of $G$. Put $\Gamma=G(A)$.

Let $S$ be any integral domain with infinitely many units. According to
Theorem \ref{thm:double centralizer} we have%
\[
U_{\alpha}(S)Z(S)=\mathrm{Z}\left(  C_{G(S)}(v)\right)
\]
whenever $1\neq v\in U_{\alpha}(S).$ This holds in particular for the rings
$S=\mathfrak{o}_{\mathfrak{p}}$. Take $u_{\alpha}\in U_{\alpha}(A)$ to have
$\mathfrak{p}$-component $x_{\alpha}(1)$ for each $\mathfrak{p}\in\mathcal{P}$
(or every $\mathfrak{p}$ when $A=\mathbb{A}$); then%
\[
U_{\alpha}(A)Z(A)=\mathrm{Z}\left(  C_{G(A)}(u_{\alpha})\right)  .
\]
Given this, the proof of Corollary \ref{defibl} now shows that $U_{\alpha}(A)$
is a definable subgroup of $\Gamma$ (the result is stated for integral domains
but the argument remains valid, noting that in the present case $A/2A$ is finite).

Associated to each root $\alpha$ there is a morphism $\varphi_{\alpha
}:\mathrm{SL}_{2}\rightarrow G$ sending $u(r)$ to $x_{\alpha}(r)$ and $v(r)$
to $x_{-\alpha}(-r)$ (see \cite{C}, Chapter 6 or \cite{St}, Chapter 3). This
morphism is defined over $\mathbb{Z}$ and satisfies
\[
K_{\alpha}:=\mathrm{SL}_{2}(A)\varphi_{\alpha}\leq G(A).
\]

\begin{lemma}
\label{Klemma}$K_{\alpha}=U_{-\alpha}(A)U_{\alpha}(A)U_{-\alpha}(A)U_{\alpha
}(A)U_{-\alpha}(A)U_{\alpha}(A)U_{-\alpha}(A)U_{\alpha}(A).$
\end{lemma}

\begin{proof}
This follows from the corresponding identity in $\mathrm{SL}_{2}(A),$ which in
turn follows from (\ref{h-identity}) and the fact that $w=uvu$.
\end{proof}

\bigskip

We may thus infer that each $K_{\alpha}$ is a definable subgroup of $G(A)$.
Fixing a root $\gamma$, we identify $A$ with $U_{\gamma}(A)$ by $r\longmapsto
r^{\prime}=x_{\gamma}(r).$ Proposition \ref{P1} now shows that $A$ is
definable in $G(A)$.

As above, $G(A)$ is $A$-definable as a set of $d\times d$ matrices that
satisfy a family of polynomial equations over $\mathbb{Z}$, with group
operation matrix multiplication.

To complete the proof we need to establish

\medskip\textbf{Step }$\mathbf{1}^{\mathbf{\prime}}$\textbf{:} The ring
isomorphism $A\rightarrow U_{\gamma}(A);~r\longmapsto r^{\prime}=x_{\gamma
}(r)\in\mathrm{M}_{d}(A)$ is definable in ring language. This follows from
(\ref{root_matrix}) in \S \ref{bisection}.\bigskip

\textbf{Step }$\mathbf{2}^{\prime}$\textbf{:} The group isomorphism
$\theta:G(A)\rightarrow G(A^{\prime})\subseteq\mathrm{M}_{d}(U_{\gamma}(A))$
is definable in group language.

\medskip

To begin with, Lemma \ref{def_theta} shows that for each root $\alpha$, the
restriction of $\theta$ to $U_{a}(A)$ is definable (this is established for
$A$ an integral domain, but the proof is valid in general). Next, we observe
that $G(A)$ has finite elementary width:

\begin{lemma}
There is is finite sequence of roots $\beta_{i}$ such that%
\[
G(A)=\prod_{i=1}^{N}U_{\beta_{i}}(A).
\]

\end{lemma}

\begin{proof}
This relies on results from Chapter 8 of \cite{St}. Specifically, Corollary 2
to Theorem 18 asserts that if $R$ is a PID, then (in the above notation)
$G(R)$ is generated by the groups $K_{\alpha}$. It is clear from the proof
that each element of $G(R)$ is in fact a product of bounded length of elements
from various of the $K_{\alpha}$; an upper bound is given by the sum $N_{1}$,
say, of the following numbers: the number of positive roots, the number of
fundamental roots, and the maximal length of a Weyl group element as a product
of fundamental reflections. If the positive roots are $\alpha_{1}%
,\ldots,\alpha_{n}$ it follows if $R$ is a PID that
\[
G(R)=\left(  \prod_{j=1}^{n}K_{\alpha_{j}}\right)  \cdot\left(  \prod
_{j=1}^{n}K_{\alpha_{j}}\right)  \cdot\ldots\cdot\left(  \prod_{j=1}%
^{n}K_{\alpha_{j}}\right)  ~\ \text{(}N_{1}\text{ factors).}%
\]
As each of the rings $\mathfrak{o}_{\mathfrak{p}}$ is a PID (or a field), the
analogous statement holds with $A$ in place of $R$.

The result now follows by Lemma \ref{Klemma}, taking $N=8nN_{1}$.
\end{proof}

\bigskip

Thus $\theta$ is definable as follows: for $g\in G(A)$ and $M\in\mathrm{M}%
_{d}(U_{\gamma}(A)),$ $g\theta=M$ if and only if there exist $v_{i}\in
U_{\beta_{i}}(A)$ and $M_{i}\in\mathrm{M}_{d}(U_{\gamma}(A))$ such that
$g=v_{1}\ldots v_{N}$, $M=M_{1}\cdot\ldots\cdot M_{N}$ and $M_{i}=v_{i}\theta$
for each $i$. Here $M_{1}\cdot M_{2}$ etc denote matrix products, which are
definable in the language of $G$ because the ring operations on $A^{\prime
}=U_{\gamma}(A)$ are definable in $G$.

This completes the proof. \ $\blacksquare$

\section{Appendix \label{comfor}}

\bigskip We recall some commutator formulae (\cite{C}, Thms. 5.2.2 and 4.1.2,
or \cite{St}, Chapter 3,\ Cor. to Lemma 15). Here $\Phi$ is a root system,
$\alpha,~\beta\in\Phi$. If $\alpha+\beta\notin\Phi$ then $[x_{\alpha
}(r),x_{\beta}(s)]=1.$\ If $\alpha+\beta\in\Phi$\ then $\alpha$ and $\beta$
span a root system $\Phi_{1}$ of rank $2$ and there are three possibilities
(assuming w.l.o.g. that $\alpha$ is short, if $\alpha$ and $\beta$ are of
different lengths). Here $\varepsilon=\pm1$.

\bigskip

$\Phi_{1}=A_{2}:$%
\begin{align*}
\lbrack x_{\alpha}(r),x_{\beta}(s)]  &  =x_{\alpha+\beta}(\varepsilon rs)\\
\lbrack x_{-\alpha}(r),x_{\alpha+\beta}(s)]  &  =x_{\beta}(\varepsilon rs)
\end{align*}

$\Phi_{1}=B_{2}:$%
\begin{align*}
\lbrack x_{\alpha}(r),x_{\beta}(s)]  &  =x_{\alpha+\beta}(\varepsilon
rs)x_{2\alpha+\beta}(\pm r^{2}s)\\
\lbrack x_{\alpha}(r),x_{\alpha+\beta}(s)]  &  =x_{2\alpha+\beta}(\pm2rs)\\
\lbrack x_{-\alpha}(r),x_{\alpha+\beta}(s)]  &  =x_{\beta}(\pm2rs)\\
\lbrack x_{-\alpha}(r),x_{2\alpha+\beta}(s)]  &  =x_{\alpha+\beta}(\pm
rs)x_{\beta}(\pm r^{2}s)\\
\lbrack x_{\alpha+\beta}(r),x_{-\beta}(s),]  &  =x_{\alpha}(\varepsilon
rs)x_{2\alpha+\beta}(\pm r^{2}s)
\end{align*}
$\Phi_{1}=G_{2}:$%
\begin{align*}
\lbrack x_{\beta}(r),x_{\alpha}(s)]  &  =x_{\alpha+\beta}(\varepsilon
rs)x_{2\alpha+\beta}(-\varepsilon rs^{2})x_{3\alpha+\beta}(-rs^{3}%
)x_{3\alpha+2\beta}(\pm r^{2}s^{3})\\
\lbrack x_{\alpha+\beta}(r),x_{a}(s)]  &  =x_{2\alpha+\beta}(-2rs)x_{3\alpha
+\beta}(-3\varepsilon rs^{2})x_{3\alpha+2\beta}(\pm3r^{2}s)\\
\lbrack x_{2\alpha+\beta}(r),x_{a}(s)]  &  =x_{3\alpha+\beta}(3\varepsilon
rs)\\
\lbrack x_{\alpha+\beta}(r),x_{-\beta}(s)]  &  =x_{\alpha}(-\varepsilon
rs)x_{2\alpha+\beta}(\pm r^{2}s)x_{3\alpha+2\beta}(\pm r^{3}s)x_{3\alpha
+\beta}(\pm r^{3}s^{2})
\end{align*}
(There are other possible combinations of signs, depending on the choice of
Chevalley basis. We assume for convenience that the basis is chosen so as to
obtain this particular form for the commutator formulae.)

\section{Acknowledgements}

We would like to thank Matthias Aschenbrenner, Brian Conrad, Jamshid
Derakhshan, Philippe Gille, Bob Guralnick, Franziska Jahnke, Martin Liebeck,
Evgeny Plotkin, Gopal Prasad, Donna Testermann, Nikolai Vavilov, Richard Weiss
and Boris Zilber for their valuable input on various questions.

\end{document}